\documentclass[11pt]{article}

\usepackage[utf8]{inputenc}

\usepackage[margin=1in]{geometry}

\usepackage[numbers]{natbib}

\usepackage{graphicx} 

\usepackage{amsmath,amsthm,amssymb,hyperref}
\usepackage{physics}
\usepackage{amsfonts}
\usepackage{graphicx}
\usepackage{mathdots}
\usepackage{comment, dsfont } 
\usepackage{ mathrsfs }
\usepackage{ verbatim }
\usepackage{xcolor}
\usepackage{paracol}
\usepackage{amssymb}
\usepackage{hyperref}
\usepackage{cleveref}

\usepackage{thmtools}
\usepackage{thm-restate}
 
\usepackage{mathtools}
\usepackage{subcaption} 
\usepackage{authblk}

\usepackage{tikz}
\usetikzlibrary{decorations.markings, arrows.meta}

\def\tsc#1{\csdef{#1}{\textsc{\lowercase{#1}}\xspace}}

\tsc{WGM}
\tsc{QE}
\tsc{EP}
\tsc{PMS}
\tsc{BEC}
\tsc{DE}

\newcommand{\caputo}[3][0]{{}_{#1}^{C}\! D^{#2}_{#3}} 
\newcommand{\RL}[3][0]{{}_{#1}^{RL}\! D^{#2}_{#3}}
\newcommand{\J}[3][0]{{}_{#1}\! J^{#2}_{#3}}

\newcommand{\R}{\mathbb{R}}
\newcommand{\C}{\mathbb{C}}
\newcommand{\Z}{\mathbb{Z}}

\newcommand{\El}{\mathcal{L}}
\newcommand{\Ef}{\mathcal{F}}

\newtheorem{theorem}{Theorem}
\newtheorem{lemma}[theorem]{Lemma}

\theoremstyle{definition}
\newtheorem{definition}[theorem]{Definition}

\theoremstyle{remark}
\newtheorem{rmk}{Remark}

\begin{document}
\let\WriteBookmarks\relax
\def\floatpagepagefraction{1}
\def\textpagefraction{.001}


\title{Anomalous diffusion memory factorization: Characteristic timescales and application to inverse problem}


\author[1,2]{William Cvetko \thanks{wcvetko@umd.edu}}
\author[2]{Elena Cherkaev\thanks{elena@math.utah.edu}}
\affil[1]{University of Maryland, Department of Physics}
\affil[2]{University of Utah, Department of Mathematics}

\maketitle

\begin{abstract}
Memory effects and anomalous diffusion arising in models of transport in complex media, fractals, and viscoelastic materials can be described by fractional-differential equations, such as the time-fractional diffusion equation \(D^\alpha_t u -  \Delta_x u = 0\).
The paper develops a decomposition of solutions of these equations into a product of spatial and temporal components, corresponding to  a freeze-out at long times.
For fractional diffusion with the Caputo derivative
the spatial factor is the inverse-Laplacian $(-\Delta)^{-1}$ of the initial data, while for the Riemann-Liouville derivative
it is the inverse bi-Laplacian $ (-\Delta)^{-2}$. In both cases, the temporal factor of the solution is a scaled negative power of time. This memory artifact explicitly encodes the initial data, which  gives a simple and robust way to reconstruct the initial conditions in the backward-in-time inverse problem with solution values measured at long times.
We derive characteristic timescales for Mittag-Leffler functions, which correspond to such factorization in anomalous diffusion. This also enables an accurate approximation of the number of real zeros of the Mittag-Leffler function. 
We apply these results to heat transfer on a comb, a model which manifests subdiffusion arising from the comb's fractal structure.

\end{abstract}

\section{Introduction}

Anomalous diffusion equations model transport phenomena in complex hierarchical networks and fractal structures, from flow within percolation clusters and porous rocks \cite{bakuninMultiscalePercolationScaling2004,raghavanFractionalDerivativesApplication2011} to homogenized multiscale media \cite{benarousMultiscaleHomogenizationBounded2003,armstrongAnomalousDiffusionFractal2025} and continuous-time random walks \cite{metzlerRandomWalksGuide2000}. 
A particular example that we explore is transport within a comb fractal medium, which has found application across different fields serving as a model for the motion of ions along spiny dendrites in nerve cells \cite{mendezComblikeModelsTransport2013}, fluid flow in a porous medium \cite{arkhincheevAnomalousDiffusionDrift1991}, 
or spreading of wave turbulence 
\cite{milovanovTurbulenceSpreadingAnomalous2025} among others \cite{iominFractionalDynamicsComblike2018,sandevHeterogeneousDiffusionComb2018}.
In analyzing these fractional models an important role is played by eigenfunctions of fractional derivatives \cite{grigolettoLinearFractionalDifferential2018,cvetkoConvolutiontosumIdentitiesMittagLeffler2026}, which are characterized by Mittag-Leffler functions
\cite{mainardiFractionalRelaxationoscillationFractional1996,mainardiWhyMittagLefflerFunction2020,mainardiWhyMittagLefflerFunction2020,hauboldMittagLefflerFunctionsTheir2011,gorenfloMittagLefflerFunctionsRelated2014,gorenfloRecentAdvancesTheory2009,mieghemMittagLefflerFunction2021}. 
In particular, each Fourier mode of time-fractional diffusion solutions evolves according to a Mittag-Leffler function in time.
%
%
Though analogous in many respects to exponentials and sinusoids, Mittag-Leffler functions are distinctive in that they can have \textit{long tails} \cite{mainardiPropertiesMittagLefflerFunction2014,mainardiMittagLefflertypeFunctionsFractional2000}. The Mittag-Leffler function,
\[
e_{\alpha}(t):=E_{\alpha}(-t^\alpha) := \sum_{n=0}^\infty \frac{(-t^\alpha)^n}{\Gamma( \alpha n+1)},
\]
   asymptotically approaches $\frac{t^{-\alpha}}{\Gamma(1-\alpha)}$, decaying according to a power law.
   This behavior is well-known for $\alpha \in (0,1)$, it also occurs for $\alpha \in (1,2)$, and for the two-parameter generalization 
   \[
   e_{\alpha,\beta}(t) := t^{\beta-1}E_{\alpha,\beta}(-t^\alpha) := t^{\beta-1} \sum_{n=0}^\infty \frac{(-t^\alpha)^{n}}{\Gamma(n\alpha+\beta)}
   \]
   asymptotically approaching $\frac{t^{-\alpha+\beta}}{\Gamma(-\alpha+\beta)}$ \cite{parisAsymptoticsMellinBarnesIntegrals2001,erdelyiMiscellaniousFunctions1955}. 
Owing to these long tails, solutions to the anomalous diffusion equation (for domain $\Omega \subseteq \R^n$)
 \begin{equation}
 	D^\alpha_t u -  \Delta u = 0 \label{eq:really_general_anom_diff} \qquad \alpha\in(0,1)\cup(1,2), \qquad (x,t) \in \Omega \cross (0,\infty),
 \end{equation}
have notable asymptotic properties not found in standard diffusion \cite{liAsymptoticsSolutionsSuperdiffusion2023,liInitialboundaryValueProblems2023,chengAsymptoticBehaviorSolutions2017,maAsymptoticsSolutionsAnomalous2013,metzlerRandomWalksGuide2000,kimAsymptoticBehaviorsFundamental2016}, including bounds such as
\begin{equation}     \Vert u(\cdot,t) \Vert < C t^{-\alpha}.\end{equation}
 Additionally, this long-tailed algebraic decay makes the inverse problem less ill-posed than for the standard heat equation \cite{jinTutorialInverseProblems2015,sakamotoInitialValueBoundary2011,floridiaWellposednessBackwardProblems2020,alimovBackwardProblemsTime2021,liuBackwardProblemTimefractional2010,floridiaBackwardProblemsTime2020,zhangNumericalAnalysisBackward2020}.

In this paper we demonstrate that solutions to the anomalous diffusion equation,  \cref{eq:really_general_anom_diff}, for both Caputo and Riemann-Liouville derivatives, 
$\alpha \in (0,1)$ and $\alpha \in (1,2),$ all experience a form of 'freeze-out', where the solution factorizes into space and time components as $t\rightarrow \infty$.
In each case, the spatial component depends only on the initial data and geometry of the domain, while the time component depends only on the order of the time derivative  $\alpha$.
For example, in the Caputo case with $\alpha\in (0,1)$ and initial conditions $u(x,t=0) = v(x) $, the solution can be represented as
\begin{equation}
    u(x,t)\approx \frac{t^{-\alpha}}{\Gamma(1-\alpha)} (-\Delta)^{-1}[v](x),
\end{equation}
where the inverse Laplacian is defined by
\[(-\Delta)^{-1}[f](x) = \int _\Omega f(y) K(x,y) dy , \] and $K$ is the Green's function for Poisson's equation in the domain $\Omega$.
 This factorization of the solution can be exploited for the reconstruction of the initial data for the backward-in-time inverse problem, even at asymptotically large times.
 
The present work is organized as follows. In \cref{section:definitions} we define and derive the solutions to the time-fractional diffusion equation, including subdiffusion ($0<\alpha<1$) and  superdiffusion ($1<\alpha<2$) for both Caputo and Riemann-Liouville types of fractional derivative, and different types of initial data.  
In \cref{section:asymptotics} we derive and justify the 
long-time factorized form of these solutions, provide a simple formula for recovering the initial data at asymptotic times, and show the mechanism behind the form they take (including how these results generalize for anomalous diffusion with various memory kernels). In \cref{section:Characteristic_time_scales} we use a different representation of the Mittag-Leffler function to derive characteristic timescales. In \cref{section:zereos_of_ml} we use characteristic times in a novel method for counting the real zeros of the Mittag-Leffler function $e_{\alpha,\beta}(t)$ for a given $\alpha,\beta.$
In \cref{section:comb_diffusion} we explore
diffusion on a comb, in which both Caputo and Riemann-Liouville subdiffusion describe aspects of transport in a fractal structure.


\section{Basics of anomalous diffusion} \label{section:definitions}
This paper considers solutions to the time-fractional diffusion equation,
\begin{equation} \label{eq:bare_diffusion_equation}
    D^\alpha_t u(t,x) - \Delta_x u(t,x) = 0 \qquad 0<\alpha<2
\end{equation}
with associated initial data and boundary conditions. 
There are multiple conventions for fractional derivatives, with distinct solutions and types of initial data corresponding to the particular fractional derivative in use \cite{hilferFractionalDiffusionBased2000,mainardiFractionalRelaxationoscillationFractional1996,metzlerRandomWalksGuide2000}. The two formulations of fractional derivative used in this paper, Caputo ($\caputo[]{\alpha}{t})$ and Riemann-Liouville ($\RL[]{\alpha  }{t}$), each rely on the Riemann-Liouville fractional integral:
 \begin{definition}
\label{def:fractionalIntegral}
Provided it converges, the order-$\nu \, (\nu>0)$ Riemann-Liouville fractional integral of $f(t)$
is defined as
\begin{equation}
    \J[]{\nu}{} [f](t) :=  \frac{1}{\Gamma(\nu)} \int _{0}^t \!\!f(z) (t-z)^{\nu-1} dz \label{eq:fractionalIntegral}
\end{equation}
For $\nu=0$, $J^\nu$ is taken to be the identity operator.
\end{definition}
The Caputo and Riemann-Liouville derivatives of order $\alpha=n-\nu$ are defined by composing an order-$n$ derivative with an order-$\nu$ integral, differing by the order of composition.

 \begin{definition}
 \label{def:derivativeRL}
Provided the fractional integral converges, the order-$\alpha$ ($\alpha>0$) Riemann-Liouville and Caputo fractional derivatives of $f(t)$ are respectively defined as
     \begin{align}
    \RL[]{\alpha}{}[f](t) :=& D^{\lceil \alpha \rceil} \J[]{\lceil\alpha\rceil-\alpha}{} [f](t)
    \label{eq:derivativeRL}
    \\
    \caputo[]{\alpha}{}[f](t) :=&  \J[ ]{\lceil\alpha\rceil-\alpha}{} \left [D^{\lceil \alpha \rceil}f \right](t),
    \label{eq:derivativeCaputo}
\end{align}
where $\lceil\alpha\rceil$ is the ceiling of $\alpha$ ($\alpha$ rounded up to the next integer). For $\alpha\leq0$, both derivatives are taken to be $\J[]{ -\alpha}{}[f](t).$
\end{definition}

Fractional differential equations (FDEs) with Caputo and Riemann-Liouville derivatives admit different types of initial data. Loosely speaking, solutions to Caputo FDEs are specified by their initial value and (integer-order) derivatives, while solutions to Riemann-Liouville FDEs require a form of fractional initial data (see \cite{heymansPhysicalInterpretationInitial2006}), and can diverge as $t\rightarrow0^+$.

For $\alpha\in (0,1)$, \cref{eq:bare_diffusion_equation} describes \textit{subdiffusion}, and \textit{superdiffusion} for $\alpha\in (1,2)$.  Owing to the second-order time derivative embedded in $D^\alpha_t$ for $\alpha\in(1,2)$, for superdiffusion (like for the wave equation) there are two sets of initial data to be specified, as opposed to the one set of initial data to be specified for subdiffusion. To handle these distinctions, we now label four solutions to the fractional diffusion equation for the rest of the paper--corresponding to the two formulations of fractional derivative (Caputo/Riemann-Liouville) with two possible types of initial data for each.

\begin{table}[]

    \label{tab:u_xi}
   \renewcommand{\arraystretch}{1.25}
\begin{tabular}{|c|c|c|c|c|}
     \hline
      $\gamma$    &Type    &   Range of $\alpha$   & Initial condition &  Other initial condition
    \\ \hline 
     $0$ &Caputo & $(0,1)$ &  $u_{0}(x,0)=v_0(x)$   & \text{N/A}
     \\ \hline 
   $\alpha-1$ & RL & $(0,1)$ &  $\RL[]{\alpha-1}{t} [u_{\alpha-1}](x,0)=v_{\alpha-1}(x)$  & \text{N/A}
    \\ \hline 
    $1$ & Caputo & $(1,2)$ &  $\partial_tu_{1}(x,0)=v_1(x)$ & $u_{1}(x,0) \equiv 0$
     \\ \hline 
     $0$ &Caputo & $ (1,2)$ &  $u_{0}(x,0)=v_0(x)$   &  $\partial_t u_0(x,0) \equiv0$
   \\ \hline 
   $\alpha-1$ & RL & $ (1,2)$ &  $\RL[]{\alpha-1}{t} [u_{\alpha-1}](x,0)=v_{\alpha-1}(x)$  & $\RL[]{\alpha-2}{t} [u_{\alpha-1}](x,0) \equiv0 $
   \\ \hline 
    $\alpha-2$ & RL & $(1,2)$ &  $\RL[]{\alpha-2}{t} [u_{\alpha-2}](x,0)=v_{\alpha-2}(x)$ & $ \RL[]{\alpha-1}{t} [u_{\alpha-2}](x,0)\equiv 0  $
     \\ \hline 
    \end{tabular}
    \caption{A summary of the cases of time-fractional diffusion considered in this paper, as specified in \cref{def:anom_diff_u_xi}.}
\end{table}

\begin{definition} \label{def:anom_diff_u_xi}
    With $\alpha \in (0,1) \cup (1,2)$, and for subscript $\gamma \in \{ 0,1,\alpha-1,\alpha-2\}$, $u_{\gamma}(x,t)$ denotes the unique solution to \begin{equation} \label{eq:anom_diff_u_xi}
       \left \{ 
       \begin{aligned}& 
           D_{t}^\alpha u_{\gamma} -  \Delta u_{\gamma} = 0 \quad & (x,t) \in \Omega\cross (0,\infty)
           \\
           &D^\gamma_t u_{\gamma} |_{t=0_+} = v_{\gamma}(x) & 
           \\
            &\mathcal{B}[u_\gamma] = 0  & \text{on } \partial \Omega \cross (0,\infty),
       \end{aligned}\right.
    \end{equation}
  where $D^\gamma$ are Riemann-Liouville derivatives for $\gamma=\alpha-1,\alpha-2$ (for $\gamma=0,1$ there is no ambiguity).  Further, assume $v_\gamma \in L_2(\Omega)$. For $\alpha>1$, any unspecified initial data is taken to be zero. Though $\gamma$ takes four values, there are actually six distinct cases, summarized in \cref{tab:u_xi}.
  
If $\Omega$ is bounded, we assume the boundary conditions ($\mathcal B [u_\gamma]\vert_{\partial \Omega}= 0 $) are such that  $-\Delta $ has strictly positive eigenvalues. Here we assume that $v_\gamma\in L_2(\Omega)$, and the initial data condition  $D^\gamma_t u_{\gamma} |_{t=0} = v_{\gamma}(x)$ is satisfied in $L_2(\Omega)$ 
in the $t\rightarrow0$ limit.

\end{definition}
If $\alpha\in (1,2)$ and one considers the problem with both types of initial data, they can simply be added together by linearity of \cref{eq:anom_diff_u_xi}. For $\alpha \in (0,1)$, $u_{1}$ and $u_{\alpha-2}$ are not even defined.

We introduce Mittag-Leffler functions:
\begin{definition} \label{def:mittag_leffler_2_param}
For $z\in \C$, $\alpha \in \R_{>0}$, $\beta \in \R$, the two-parameter Mittag-Leffler function is
    \begin{equation}
        E_{\alpha,\beta}(z) := \sum_{n=0}^\infty \frac{z^n}{\Gamma(\alpha n + \beta)}.
    \end{equation}
    \end{definition}
    \begin{rmk} Formally, the integral which defines
$\Gamma(y)$ does not converge for $ y \leq0$, but the function can be analytically continued into that region, with poles at the negative integers. Due to those poles, we interpret hereafter $\frac{1}{\Gamma(-j)} (j\in\Z_{\geq0})$ as being zero.
\end{rmk}

\begin{definition}
For $\alpha,t, \lambda \in \R_{>0}$,  $\beta \in \R$, we denote a parameterization of the two-parameter Mittag Leffler function as
\begin{align}
       e_{\alpha,\beta}(t;\lambda) =& t^{\beta-1} E_{\alpha,\beta}(-\lambda t^\alpha) = \sum_{n=0}^\infty \frac{(-\lambda)^n t^{\alpha n + \beta-1}}{\Gamma(\alpha n + \beta)} ,\label{eq:series_def_e_ab_lambda}
\end{align}
 which has a Laplace transform of 
\begin{equation} \label{eq:laplace_of_e}
    \El \left[e_{\alpha,\beta}(t;\lambda)\right](s) = \frac{s^{\alpha-\beta}}{s^\alpha+\lambda} \qquad (\beta<\alpha+1).
\end{equation}
For $\lambda=1$ we omit the parameter $\lambda$, so $e_{\alpha,\beta}(t;1)  := e_{\alpha,\beta}(t).$
\end{definition}

\begin{rmk} 
There are other specialized parameterizations of Mittag-Leffler functions  \cite{kilbasTheoryApplicationsFractional2006,lorenzoGeneralizedFunctionsFractional2008,garrappaNumericalEvaluationTwo2015,cvetkoConvolutiontosumIdentitiesMittagLeffler2026}. In this work we follow the convention used by Mainardi and Gorenflo in \cite{mainardiFractionalCalculusWaves2010,gorenfloFractionalCalculusIntegral2008,mainardiMittagLefflertypeFunctionsFractional2000,gorenfloMittagLefflerFunctionsRelated2014}. We will often refer to $e_{\alpha,\beta}(t;\lambda)$ as a 'Mittag-Leffler function', which we hope will not cause ambiguity.
\end{rmk}

The Mittag-Leffler functions $e_{\alpha,\beta}(t)$ and $e_{\alpha,\beta}(t;\lambda)$, are related through
\begin{equation} \label{eq:relate_e_e_lambda}
    e_{\alpha,\beta}(t;\lambda) = t^{\beta-1} E_{\alpha,\beta}(-\lambda t^\alpha) =  \lambda^{ \frac{1-\beta}{\alpha}} e_{\alpha,\beta}(\lambda^{1/\alpha}t).
\end{equation}

The Mittag-Leffler function provides the solutions to the fundamental fractional differential equation for both Caputo and Riemann-Liouville derivatives \cite{cvetkoConvolutiontosumIdentitiesMittagLeffler2026}.
\begin{lemma}
\label{lemma:solution_to_fundamental_equation}
    For $\alpha,t,\lambda>0$, the solution to 
    \begin{equation}
        \left \{  \begin{aligned}
            \caputo[]{\alpha}{t} [f](t) =& -\lambda f(t)
            \\
           \frac{d^\ell f}{dt^\ell}\bigg \vert _{t=0} =& b_{\ell} \qquad \text{for } \ell \in \{0,1,\cdots,\lceil\alpha\rceil-1\}
        \end{aligned} \right.
    \end{equation}
    is given by \begin{equation}
        f(t) = \sum_{\ell=0}^{\lceil \alpha \rceil -1}  b_{\ell}\,  e_{\alpha,\ell+1}(t;
        \lambda).
    \end{equation}
    The solution to
    \begin{equation}
        \left \{  \begin{aligned}
           & \RL[]{\alpha}{t} [f](t) =  -\lambda f(t)
            \\
            &\RL[]{\alpha-\ell}{t}[f] \vert _{t=0^+} = b_{\alpha-\ell} \qquad \text{for } \ell \in \{1,\cdots,\lceil\alpha\rceil\}
        \end{aligned} \right.
    \end{equation}  
    is given by
\begin{equation}
        f(t) = \sum_{\ell=1}^{\lceil \alpha \rceil }  e_{\alpha,\alpha+1-\ell}( t;\lambda) b_{\alpha-\ell}.
    \end{equation}
    
\end{lemma}

\begin{proof}
  Take $\alpha = n-\nu$, so that $\caputo[]{\alpha}{} = \J[]{\nu}{} \frac{d^n}{dt^n}$, and $\RL[]{\alpha}{} = \frac{d^n}{dt^n} \J[]{\nu}{}$ . 
  The Laplace transform $\El[ \J[]{\nu}{} f](s) = s^{-\nu} \El[f](s)$ (see \cite{cvetkoConvolutiontosumIdentitiesMittagLeffler2026}).
  Now, consider the Laplace transform of the Caputo equation,
  \begin{equation*}
       \El \left[ \J[]{\nu}{} \frac{d^n}{dt^n} f \right](s) = -\lambda \El[f](s).
  \end{equation*}
  Using properties of the Laplace transform,  this becomes
  \begin{align*}
       s^{-\nu} \left ( s^n \El[f] - \sum_{\ell=0}^{n-1} s^{n-1-\ell} \frac{d^\ell f}{dt^\ell} \bigg \vert_{t=0} \right) =& -\lambda \El[f](s)
      \\
      \El[f] =&  \sum_{\ell=0}^{n-1} b_{\ell} \frac{ s^{\alpha-1-\ell} }{s^\alpha + \lambda}
  \end{align*}
and the result follows from \cref{eq:laplace_of_e}. For the Riemann-Liouville, we proceed in a similar manner.
\begin{align*}
    \El \left[ \frac{d^n (\J[]{\nu}{}f)}{dt^n}
    \right](s) =& -\lambda \El[f](s)
    \\
    s^n \El[ \J[]{\nu}{}f](s) - \sum_{\ell=0}^{n-1} s^{n-1-\ell}  \frac{d^\ell (\J[]{\nu}{}f)}{dt^\ell} \bigg |_{t=0} =& -\lambda \El[f](s)
    \\
    s^{\alpha} \El[  f](s) - \sum_{\ell=0}^{n-1} s^{n-1-\ell}   \RL[]{\ell-\nu}{}[f]|_{t=0} =& -\lambda \El[f](s),
\end{align*}
and the rest is rearranging.
\end{proof}

Now we can explicitly represent the solutions of fractional diffusion, $u_\gamma$, with Mittag-Leffler functions.

\begin{theorem}\label{thm:spectral_rep_u_xi}
    Let $u_{\gamma}(x,t)$ ($\gamma=0,1,\alpha-1,\alpha-2$) be specified by \cref{def:anom_diff_u_xi}. The Fourier $ \left(f(x) \xrightarrow[]{\Ef} \tilde f(k) \right)$  representation of $u_\gamma$ is given by
    \begin{equation}
        \tilde u_{\gamma}(k,t) = \tilde v_{\gamma}(k)\, e_{\alpha,\gamma+1}(t; k^2) .
    \end{equation}
\end{theorem}
\begin{proof}
    Regardless of derivative type, the subdiffusion equation
    \begin{equation}
       \begin{cases}
            D^{\alpha}_t u_{\gamma}(x,t) -  \Delta_x u_{\gamma}(x,t) = 0
            \\
            D^{\gamma}_t[u_\gamma](x,0^+) = v_{\gamma}(x)
       \end{cases}
    \end{equation}
    is represented in the Fourier domain as
    \begin{equation}
         \begin{cases}
             D^{\alpha}_t \tilde u_{\gamma}(k,t) +  k^2 \tilde u_{\gamma}(k,t) = 0
             \\
             D^\gamma_t \tilde u_{\gamma}(k,0^+) = \tilde v_\gamma (k),
         \end{cases} 
    \end{equation}
    where we abbreviate $\abs{ k}^2$ as $k^2 $ for vector-valued $k$.
    For each value of $k^2$, this is an instance of the fundamental fractional differential equation, which has the stated solution by \cref{lemma:solution_to_fundamental_equation}. 
\end{proof}

\section{Anomalous diffusion freeze-out}
\label{section:asymptotics}
\subsection{Main results}

\begin{figure}
    \centering
    \includegraphics[width=1.\linewidth]{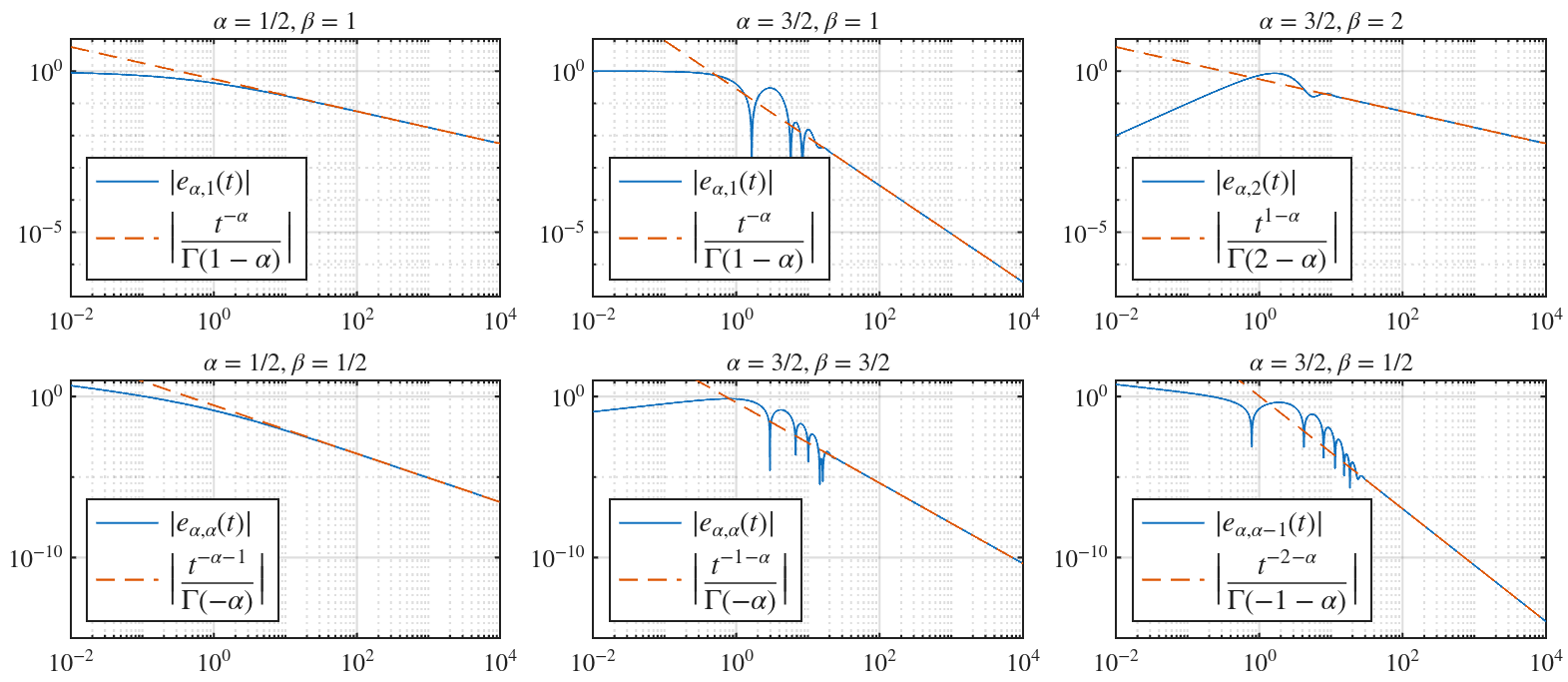}
    \caption{ The Mittag-Leffler function,
    $e_{\alpha,\beta}(t)$, is plotted (in absolute value) against its leading-order asymptotic term as found in \cref{eq:asymptotic_of_e} for the six distinct regimes of fractional diffusion under consideration. 
 }
    \label{fig:ML_asymptotics_3hlf}
\end{figure}

To analyze the long-time behavior of the $u_\gamma(x,t)$ (as specified in \cref{def:anom_diff_u_xi,tab:u_xi}) we will consider the asymptotics of the Mittag-Leffler function.
For $\alpha \in(0,1)\cup(1,2)$, the Mittag-Leffler function admits the following asymptotic series:
\begin{equation} \label{eq:asymptotic_of_e}
    e_{\alpha,\beta}( t;\lambda )= \sum_{n=1}^{N} \frac{(-1)^{n+1} t^{-\alpha n+\beta -1} }{ \lambda ^{  n}\Gamma(\beta-\alpha n) } 
    + \mathcal O (t^{-\alpha (N+1)+\beta-1})
    \end{equation}
(see \cite{parisAsymptoticsMellinBarnesIntegrals2001,erdelyiMiscellaniousFunctions1955}) though sometimes this series is only stated for $\alpha<1$ and $\beta=1$ \cite{mainardiPropertiesMittagLefflerFunction2014}.
The leading term is generally given by $n=1$, so that for large $t$
\begin{equation}
    e_{\alpha,\beta}(t;\lambda) \approx \lambda^{-1}\frac{t^{-\alpha+\beta-1}}{ \Gamma(-\alpha+\beta)} ,
\end{equation}
except in the case where $\beta-\alpha \in \Z_{\leq0}$, where $n=2$ gives the first nonzero term:
\begin{equation}
    e_{\alpha,\alpha-j}(t;\lambda) \approx - \lambda^{-2}\frac{t^{-\alpha-1-j}}{\Gamma(-\alpha-j)} \qquad (j\in \Z_{\ge0})
\end{equation}
due to the pole at $\Gamma(-j)$. These asymptotics are displayed for $\alpha=1/2,3/2$ and relevant values of $\beta$ in \cref{fig:ML_asymptotics_3hlf}.

Applying this asymptotic approximation to $u_0$ (recall \cref{def:anom_diff_u_xi}), we have
\begin{equation}
    \tilde u_0(k,t) = \tilde v_0(k) e_{\alpha,1}(t;k^2) \approx \frac{ \tilde v_0(k)}{k^2} \frac{t^{-\alpha}}{\Gamma(1-\alpha)}.
\end{equation}
Using the inverse Fourier transform $k^2 \xrightarrow[]{\Ef^{-1}}-\Delta$, we obtain the factorized asymptotic form for $u_0(x,t)$:
\begin{equation}
    u_0(x,t) \approx (-\Delta)^{-1} [v_0](x)\frac{t^{-\alpha}}{\Gamma(1-\alpha)}. \label{eq:apx1}
\end{equation}
Similar reasoning gives the asymptotic form for the other $u_{\gamma}$:
\begin{align}
    u_1(x,t) \approx& (-\Delta)^{-1} [v_1](x)\frac{t^{-\alpha+1}}{\Gamma(2-\alpha)}\label{eq:apx2}
    \\
      u_{\alpha-1}(x,t) \approx& -(-\Delta)^{-2} [v_{\alpha-1}](x)\frac{t^{-\alpha-1}}{\Gamma(-\alpha)} \label{eq:apx3}
      \\
        u_{\alpha-2}(x,t) \approx& -(-\Delta)^{-2} [v_{\alpha-2}](x)\frac{t^{-\alpha-2}}{\Gamma(-\alpha-1)}. \label{eq:apx4}
\end{align}

Notable about \cref{eq:apx1,eq:apx2,eq:apx3,eq:apx4} is that solutions to the anomalous diffusion equation 
essentially stop evolving in time, 'freezing-out' to a static profile in space which simply decays according to a power law. Moreover, the space-dependent part (Caputo:$(-\Delta)^{-1} [v_\gamma] $, Riemann-Liouville: $(-\Delta)^{-2}[ v_{\gamma}]$) is identical for any $\alpha\in(0,1)\cup(1,2)$, and explicitly encodes the initial data. The reason why the form is different between fractional derivative types is explained in \cref{subsection:memory_kernels}. Though these asymptotic forms 
have immense implications for anomalous diffusion, to the authors' knowledge this factorization has not been investigated in previous publications.

To formalize these results, we will make use of a recurrence relation (which is implicitly present in the asymptotic series \cref{eq:asymptotic_of_e}):
\begin{lemma}[Recurrence relation] \label{lemma:recurrence_relation}
    \begin{equation} \label{eq:e_RR}
        e_{\alpha,\beta}(t;\lambda) = \frac{t^{-\alpha+\beta-1}}{\lambda \Gamma(-\alpha+\beta)} - \lambda^{-1} e_{\alpha,\beta-\alpha}(t;\lambda)
    \end{equation}
\end{lemma}
\begin{proof}
Using the series representation \cref{eq:series_def_e_ab_lambda}:
    \begin{align*}
        e_{\alpha,\beta-\alpha}(t;\lambda) =& \sum_{n=0}^\infty \frac{(-\lambda)^{n} t^{\alpha n -\alpha+\beta-1} }{\Gamma(\alpha n -\alpha+ \beta)} 
        =  \frac{t^{-\alpha+\beta-1}}{\Gamma(-\alpha+\beta)} - \lambda e_{\alpha,\beta}(t;\lambda).
    \end{align*} Rearranging this, we obtain the stated recurrence relation.
\end{proof}

Both sides of \cref{eq:apx1,eq:apx2,eq:apx3,eq:apx4} tend to zero as $t\rightarrow\infty$. We will multiply each side of these by the proper power of $t$ to show that the $L_2$ error of this approximation goes to $0$ in a bounded domain as $t\rightarrow\infty$. 
\begin{theorem} \label{thm:asymptotic_convergence_bdd_domain}
    Let
    $u_{\gamma}(x,t)$ be specified by \cref{def:anom_diff_u_xi}, with $\Omega$ being a bounded domain. Then the following limits hold in $L_2(\Omega)$ as $t\rightarrow \infty$:
    \begin{align}
          t^\alpha   u_0(x,t) \xrightarrow{L_2(\Omega)} & \frac{1}{ \Gamma(1-\alpha)} (-\Delta)^{-1} v_0(x)
        \\
          t^{\alpha-1} u_1(x,t) \xrightarrow{L_2(\Omega)} &  \frac{1}{\Gamma(2-\alpha)}(-\Delta)^{-1} v_1(x)
        \\
         t^{\alpha+1}  u_{\alpha-1}(x,t) \xrightarrow{L_2(\Omega)} &
       -\frac{1}{ \Gamma(-\alpha)}(-\Delta)^{-2} v_{\alpha-1}(x)
        \\
          t^{\alpha+2}  u_{\alpha-2}(x,t) \xrightarrow{L_2(\Omega)} & 
        -\frac{1}{ \Gamma(-1-\alpha)}(-\Delta)^{-2} v_{\alpha-2}(x).
    \end{align}
\end{theorem} 
We state and prove \cref{thm:backwards_problem} before proving \cref{thm:asymptotic_convergence_bdd_domain}.

  Note that the initial data, $v_\gamma(x)$, is explicitly encoded in this asymptotic form. This has immediate implications for the backward-in-time inverse problem of reconstructing the initial data from the values of the solution measured at later times. 
Approximate equations \cref{eq:apx1,eq:apx2,eq:apx3,eq:apx4} can be solved for the initial data.
Further, even if the current time $t$ and the order of the derivative $\alpha$ are unknown, the initial data (up to a constant) can be retrieved by Laplacian or bi-Laplacian.
A comparison of this method with the corresponding (ill-posed) problem for standard diffusion is made in \cref{subsection:compare_standard_diff}. The next theorem  shows that this approximation becomes exact in the $t\rightarrow\infty$ limit, even in an unbounded domain.
\begin{theorem}
    \label{thm:backwards_problem}
     Let 
     $u_{\gamma}(x,t)$ be specified by \cref{def:anom_diff_u_xi}, with $\Omega =\R^n$.
     Then in the limit as $t\rightarrow \infty$  \begin{align}
          t^\alpha \Gamma(1-\alpha) (-\Delta) u_{0}(x,t) \xrightarrow[]{L_2(\R^n) } &v_0(x) \label{eq:init_data_recov_0}
         \\
          t^{\alpha-1} \Gamma(2-\alpha) (-\Delta) u_{1}(x,t) \xrightarrow[]{L_2(\R^n))} &v_{1}(x) \label{eq:init_data_recov_1}
         \\
         - t^{\alpha+1} \Gamma(-\alpha) (-\Delta)^2 u_{\alpha-1}(x,t) \xrightarrow[]{L_2(\R^n))} &v_{\alpha-1}(x) \label{eq:init_data_recov_a1}
         \\
         - t^{\alpha+2} \Gamma(-\alpha-1) (-\Delta)^2 u_{\alpha-2}(x,t) \xrightarrow[]{L_2(\R^n))} &v_{\alpha-2}(x). \label{eq:init_data_recov_a2}
     \end{align}
\end{theorem}

\begin{proof}
We will prove \cref{thm:backwards_problem} first, then \cref{thm:asymptotic_convergence_bdd_domain}.

Starting with $u_0(x,t)$, by the Plancherel Theorem, $L_2$ convergence in physical space is equivalent to $L_2$ convergence in the Fourier domain. So, to show that \begin{equation}
    \lim_{t\rightarrow\infty}\Vert t^\alpha \Gamma(1-\alpha) (-\Delta) u_{0}(\cdot,t) - v_0(\cdot) \Vert_{L_2(\R^n)} = 0
\end{equation} it suffices to show that
\begin{equation}
    \lim_{t\rightarrow \infty} \int_{\R^n} 
    \left \vert 
    t^\alpha  \Gamma(1-\alpha) k^2\tilde u_0(k,t) - \tilde v_0(k)
    \right \vert^2 dk= 0.
\end{equation}
By \cref{thm:spectral_rep_u_xi}, we can factor out the initial data:
\begin{align*}
    \int_{\R^n} 
    \left \vert 
    t^\alpha  \Gamma(1-\alpha) k^2\tilde u_0(k,t) - \tilde v_0(k)
    \right \vert^2 dk =& \int_{\R^n} \left \vert  \tilde v_0(k)\right \vert^2
    \left \vert 
    t^\alpha  \Gamma(1-\alpha) k^2 e_{\alpha,1}(t;k^2) -1
    \right \vert^2 dk.
\end{align*}
Rearranging the recurrence relation \cref{eq:e_RR} with $\lambda=k^2$,
\begin{equation}
     t^{\alpha-\beta+1}  \Gamma(-\alpha + \beta) k^2 e_{\alpha,\beta}(t;k^2) =   t^{\alpha-\beta+1}  \Gamma(-\alpha + \beta) k^2 \left ( \frac{t^{-\alpha+\beta-1}}{k^2 \Gamma(-\alpha+\beta)} - k^{-2}  e_{\alpha,\alpha-\beta}(t;k^2)\right ),
\end{equation}
 and taking $\beta=1$ we get that
\begin{equation} \label{eq:unexpected_e_E_relation}
    t^\alpha  \Gamma(1-\alpha) k^2 e_{\alpha,1}(t;k^2) -1=
   - t^\alpha \Gamma(1-\alpha )e_{\alpha,1-\alpha}(t;k^2) = -\Gamma(1-\alpha) E_{\alpha,1-\alpha}(-k^2 t^\alpha).
\end{equation}
Because $ e_{\alpha,1-\alpha}(t)\propto t^{-2\alpha} $  as $t\rightarrow\infty$  (recall \cref{eq:asymptotic_of_e}), for each value of $k^2>0 $ 
    \begin{equation} \label{eq:integrand_goes_zero}
        \lim_{t\rightarrow\infty}t^{\alpha } \Gamma(1-\alpha) e_{\alpha,1-\alpha}(t;k^2) \propto \lim_{t\rightarrow\infty} t^{-\alpha} = 0
    \end{equation}
    for $\alpha\in (0,1)\cup(1,2).$
    Moreover, because $E_{\alpha,1-\alpha}(0) = \frac{1}{\Gamma(1-\alpha)} <\infty$ it follows  by continuity that $|E_{\alpha,1-\alpha}(-k^2 t^\alpha)|$ has a definite upper bound $B_0$ for $t\geq0$, independent of $k$. Thus, by the \textit{dominated convergence theorem} (with dominating function $(B_0)^2 |\tilde v_0(k) |^2$), we have that
 \begin{equation}
      \lim_{t\rightarrow \infty}\int_{\R^n} \left \vert  \tilde v_0(k)\right \vert^2
    \left \vert \Gamma(1-\alpha) E_{\alpha,1-\alpha} (-k^2 t^\alpha)  \right\vert^2 dk = \int_{\R^n}  \lim_{t\rightarrow \infty}\left \vert  \tilde v_0(k)\right \vert^2
    \left \vert \Gamma(1-\alpha) E_{\alpha,1-\alpha} (-k^2 t^\alpha)  \right\vert^2 dk  =0.
 \end{equation}
The same reasoning applies for $u_1(x,t)$, except with $\beta=2$ instead of $\beta=1$.  For the Riemann-Liouville problem, where $\gamma=\alpha-1,\alpha-2$, there is an additional step.  Focusing on $\gamma=\alpha-1$,  we must show that
\begin{align*}
     \int_{\R^n} \left \vert  \tilde v_{\alpha-1}(k)\right \vert^2
    \left \vert - t^{\alpha+1}\Gamma(-\alpha) k^4 \, e_{\alpha,\alpha}(t;k^2) - 1 \right\vert^2 dk 
\end{align*}
goes to zero.
Similar to \cref{eq:unexpected_e_E_relation}, we can use the recurrence relation (\cref{eq:e_RR}) twice to get
\begin{align*}
    - t^{\alpha+1}\Gamma(-\alpha) k^4 \, e_{\alpha,\alpha}(t;k^2) - 1  =& - t^{\alpha+1}\Gamma(-\alpha) k^4 \left ( \,0 - \frac{t^{-\alpha-1}}{ k^4\Gamma(-\alpha)}  + k^{-4} e_{\alpha,-\alpha}(t;k^2)\right )-1
    \\
    =&-\Gamma(-\alpha) E_{\alpha,-\alpha}(-k^2 t^\alpha),
\end{align*}
which, like \cref{eq:unexpected_e_E_relation}, is also bounded  for all $t\geq0$ and goes to zero for every $k^2>0$ as $t\rightarrow \infty$ -- at which point the dominated convergence theorem applies. The case of $\gamma = \alpha-2$ proceeds in the same manner.
\end{proof}

Next we prove  \cref{thm:asymptotic_convergence_bdd_domain} using the same methods.
\begin{proof}
By assumption, the Laplacian in bounded $\Omega$ has a discrete spectrum of positive eigenvalues $\{\lambda_j\}$ and orthonormal eigenfunctions $\{\phi_j(x)\}$, with $0< \lambda_0 \leq \lambda_1 \leq\cdots$. For the spectral representation of a function $f(x)$, we take \begin{equation}
    \tilde f(j) = \langle f(\cdot),\phi_j(\cdot)\rangle_{L_2(\Omega)}.
\end{equation}
Similarly to before, we apply the Plancherel/Parseval theorem in bounded domains to the case of $\gamma = 0$ (recall \cref{def:anom_diff_u_xi})
\begin{align*}
    \left \Vert t^\alpha u_0(\cdot,t)- \frac{1}{\Gamma(1-\alpha)}(-\Delta)^{-1}v_0(\cdot) \right \Vert _{L_2(\Omega)}^2
    =& \sum_{j=0 }^\infty \left \vert t^\alpha  \tilde u_0(j,t)- \frac{1}{\Gamma(1-\alpha)} (\lambda_j)^{-1} \tilde v_0(j)
    \right \vert ^2
    \\
     =& \sum_{j=0 }^\infty
     \left \vert\frac{\tilde v_0(j)}{\lambda_j \Gamma(1-\alpha)} \right \vert^2
     \left \vert t^\alpha \lambda_j \Gamma(1-\alpha)   e_{\alpha,1}(t;\lambda_j)- 1
    \right \vert ^2.
\end{align*}
The series approaches zero in the limit as $t\rightarrow\infty$, by the dominated convergence theorem  with  dominating function $(B_0)^2 \left \vert\frac{\tilde v_0(j)}{\lambda_j \Gamma(1-\alpha)} \right \vert^2$. This would also work in $\R^n$ provided that  $\frac{\tilde v_0(k) }{k^2}\in L_2(\R^n)$. 
The cases for $\gamma=1,\alpha-1,\alpha-2$ proceed in the same manner, with dominating functions proportional to  $\left \vert\frac{\tilde v_1(j)}{\lambda_j \Gamma(2-\alpha)} \right \vert^2$, $\left \vert\frac{\tilde v_{\alpha-1}(j)}{\lambda_j^2 \Gamma(-\alpha)} \right \vert^2$, and $\left \vert\frac{\tilde v_{\alpha-2}(j)}{\lambda_j^2 \Gamma(-1-\alpha)} \right \vert^2 $ respectively.    
\end{proof}

The asymptotic freeze-out described by \cref{thm:asymptotic_convergence_bdd_domain}, and reconstructions of initial data described by \cref{thm:backwards_problem}  for $\gamma=0$ are demonstrated in \cref{fig:caputo_freezeout_summary}. There we depict the solutions of $\alpha=0.9$ and $\alpha=1.75$ fractional diffusion, showing the approximations \ref{eq:apx1} to \ref{eq:apx4} 
become more accurate at long times and that even non-smooth initial data information is preserved.

\begin{figure}
     \centering
     \includegraphics[width=1.0\linewidth]{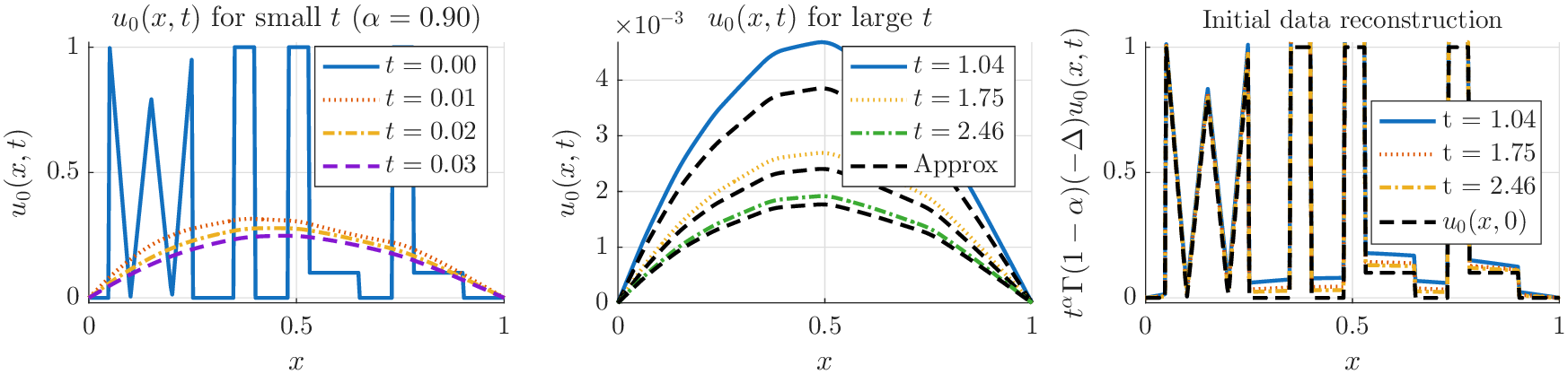}
    \\    \includegraphics[width=1.\linewidth]{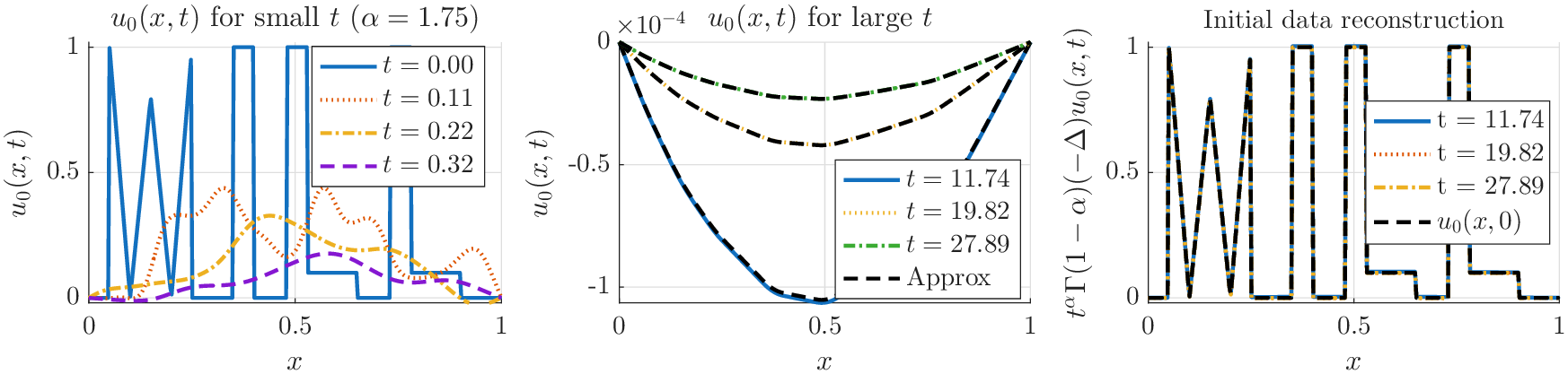}
     
     \caption{ Plots demonstrating \cref{thm:asymptotic_convergence_bdd_domain,thm:backwards_problem} for $u_\gamma(x,t)$ and $\gamma=0$ (see \cref{def:anom_diff_u_xi}), 
      with Dirichlet boundary conditions at $x=0$ and $x=1$. The top row corresponds to $\alpha = 0.9$ subdiffusion, and the bottom row corresponds to $\alpha=1.75$ superdiffusion. The left panels depict the solution for time $t=0$ (the initial data) and a few early moments.  The middle panels depict the same solution at later times, plotted with the approximation $\frac{t^{-\alpha}}{\Gamma(1-\alpha)} (-\Delta)^{-1} v_0(x)$ (dashed lines).
     The right panels show $t^{\alpha}\Gamma(1-\alpha)(-\Delta) u_0(x,t)$ at a few later times, showing the reconstruction of the initial data at later times (see \cref{fig:noisy_reconstruction} for initial data reconstruction with noise). 
     }
     \label{fig:caputo_freezeout_summary}
 \end{figure}

\subsection{Underlying memory kernel mechanism}
\label{subsection:memory_kernels}

The proofs above obscure the mechanisms behind this asymptotic freeze-out relying only on the already known asymptotic behavior of the Mittag-Leffler function. Why does the Caputo memory artifact feature $(-\Delta)^{-1}$ while the Riemann-Liouville features $(-\Delta)^{-2}$? Why does the time-dependence take the form it does? Can these memory effects be found if a different memory kernel is used? 

Here, we examine the asymptotic properties of anomalous diffusion with a more general memory kernel.
A causal convolution is defined as
\begin{equation}
    (f*g) (t) = \int_0^t f(\tau) g(t-\tau)d\tau,
\end{equation}
with the property that
\begin{equation}
    \El[f*g](s) = \El[f](s) \El[g](s).
\end{equation}
The fractional integral, given in \cref{def:fractionalIntegral}, is a convolution against a power kernel:
\begin{equation}
    \J[]{\nu}{}[f](t) := \frac{1}{\Gamma(\nu)}\int_0^t f(\tau) (t-\tau)^{\nu-1}d\tau = \frac{(\cdot)^{\nu-1}}{\Gamma(\nu)} * f.
\end{equation}
We generalize fractional diffusion by replacing $\frac{t^{\nu-1}}{\Gamma(\nu)}$ with a more generic memory kernel $\eta(t)$. These generalizations, among others, are explored in \cite{sandevGeneralizedDiffusionwaveEquation2019,al-refaiGeneralisingFractionalCalculus2023,kemppainenDecayEstimatesTimefractional2014} and elsewhere.

We consider a Caputo-type memory kernel anomalous diffusion equation as
\begin{equation} \label{eq:caputo_eta_diffusion_def}
    \left[\int_0^t \eta(t-\tau)  D_\tau [u](x,\tau) d\tau \right] - \Delta u(x,t) = 0 
\end{equation}
with $u(x,0) = v_0(x)$.
(Caputo-type means that the time derivative is applied to $u$ prior to the convolution.) Taking Fourier in $x$ ($x \rightarrow k$) and Laplace in $t$ ($t \rightarrow s)$ transforms of \cref{eq:caputo_eta_diffusion_def} (with $\eta(t) \xrightarrow{\El}M(s)$) gives the following:
\begin{equation}
    M(s) (s\tilde U(k,s) - \tilde v_0(k)) + k^2 \tilde U(k,s) = 0.
\end{equation}
where $M(s)$ is the Laplace transform of $\eta$ and 
$\tilde U(k,s)$ is Fourier-Laplace transform of $u$.
Solving for $\tilde U(k,s)$ gives
\begin{equation} \label{eq:caputo_eta_solution}
    \tilde U(k,s) = \tilde v_0(k)\frac{M(s)}{sM(s) +k^2}.
\end{equation}
Ignoring any poles, the long-time behavior of $\tilde u(k,t)$ depends on $M(s)$ near $s=0$. 
For reasonable $\eta(t)$, it holds that 
\begin{equation}
    \lim_{s\rightarrow0^+} s M(s)  = \lim_{t\rightarrow\infty} \eta(t).
\end{equation}
Consider three basic limiting cases of $\eta(t)$ as $t\rightarrow \infty$ with limits equal to $ 0, R,\infty$.  Expanding $\frac{M(s)}{sM(s) +k^2}$ to a few terms in $s$ for each case gives
\begin{equation}
    \tilde U_{caputo}(k,s) \approx \tilde v_0(k)\begin{cases}
        k^{-2}  M(s) - s M(s)^2 k^{-4} + \mathcal O(s^2M^3 k^{-6})
        &\text{for }\lim_{t\rightarrow\infty} \eta(t) = 0
        \\
         \frac{M(s)}{R+k^2} - \frac{M(sM-R)}{(R+k^2)^2 }+\mathcal O ( (sM-R)^2M)
         &\text{for }\lim_{t\rightarrow\infty} \eta(t) = R
         \\
         s^{-1} - \frac{k^2}{s^2 M(s)}+ \mathcal{O}(\frac{k^4}{s^3M^2})
         &\text{for }\lim_{t\rightarrow\infty} \eta(t) = \infty
    \end{cases}
\end{equation}
Conversely, consider a Riemann-Liouville type memory kernel anomalous diffusion equation (time derivative taken \textit{after} memory kernel convolution):
\begin{equation} \label{eq:RL_eta_anom_diffusion}
\left \{
  \begin{aligned}    
     & D_t \left[\int_0^t \eta(t-\tau)u(x,\tau) d\tau \right] - \Delta u(x,t) = 0 
      \\
      &\lim_{t_\rightarrow0^+} \int_0^t \eta(t-\tau) u(x,\tau) d\tau = v_\eta(x).
  \end{aligned} \right.
\end{equation}
Note that for nonzero $v_{\eta}(x)$, $u(x,t)$ or $\eta(t)$ must be singular as $t\rightarrow0$.
Applying Fourier and Laplace transforms,
\begin{align*}
    0 
    =&\El[ D_t \eta*\tilde u] + k^2 \El{\tilde u}
    \\
    0=&s M(s) \tilde U (k,s) - \tilde v_{\eta}(k) + k^2 \tilde U
    \\
    \tilde U(k,s) =& \tilde v_\eta(k) \frac{1}{sM(s) + k^2}.
\end{align*}
This form is similar to \cref{eq:caputo_eta_solution}, but without $M(s)$ in the numerator. So for $s\rightarrow0$ the Riemann-Liouville type solution behaves as
\begin{equation}
    \tilde U_{RL}(k,s) \approx \tilde v_\eta(k)\begin{cases}
        k^{-2}   - s M(s)  k^{-4} + \mathcal O(s^2M^2 k^{-6})
        &\lim_{t\rightarrow\infty} \eta(t) = 0
        \\
         \frac{1}{R+k^2} - \frac{sM-R}{(R+k^2)^2 }+\mathcal O ( (sM-R)^2)
         &\lim_{t\rightarrow\infty} \eta(t) = R
         \\
         \frac{1}{sM(s)}- \frac{k^2}{(s M(s))^2}+ \mathcal{O}(\frac{k^4}{(sM)^3})
         &\lim_{t\rightarrow\infty} \eta(t) = \infty.
    \end{cases}
\end{equation}

Taking $\zeta(t) = \El^{-1}[ s^{-1} M(s)^{-1}](t)$, and recalling that $\El^{-1}[1] = \delta(t)$ (and hence does not contribute as $t\rightarrow \infty)$, we obtain the leading asymptotic terms in these solutions 
in $(x,t)$ as $t\rightarrow\infty$; they are given in \cref{tab:mem_ker_lims}.
\begin{table}[]

    \label{tab:mem_ker_lims}
 \renewcommand{\arraystretch}{1.25}
\begin{tabular}{|c|c|c|}
      \hline   $\lim_{t\rightarrow \infty} \eta(t)$ & Caputo asymptotic & Riemann-Liouville asymptotic   \\
        \hline  $0$&$u(x,t) \rightarrow \eta(t) (-\Delta)^{-1} v_0(x)$& $ u(x,t) \rightarrow -\dot \eta(t)(-\Delta)^{-2}v_\eta(x)$ 
        \\
        \hline $R$ & $u(x,t)\rightarrow \eta(t) (R-\Delta)^{-1} v_0(x)$ & $u(x,t) \rightarrow-  \dot \eta(t) (R -\Delta)^{-2} v_\eta(x) $
        \\
        \hline $+\infty$ & $u(x,t) \rightarrow  v_0(x)$ & $u(x,t) \rightarrow \zeta(t) v_\eta(x)$
        \\
    \hline
\end{tabular}
 \caption{Asymptotic behavior  (not accounting for any pole contributions) for solutions of anomalous diffusion with memory kernel $\eta(t)$ for the memory-kernel generalizations of Caputo (see \cref{eq:caputo_eta_diffusion_def}) and Riemann-Liouville (see \cref{eq:RL_eta_anom_diffusion}) problems. $\zeta(t)$ is defined by $\zeta(t) = \El^{-1}[ s^{-1} M(s)^{-1}](t)$, where $M(s) = \El[\eta](s).$ }
\end{table}
So long as the memory kernel $\eta(t)$ does not reach $0$ in finite time, the relations in \cref{tab:mem_ker_lims} can be inverted to recover the initial data, see \cref{thm:backwards_problem}.

The derived asymptotic factorization described in this paper is not unique to fractional diffusion but is a memory effect which generalizes to a broader class of problems with convolution memory kernels. 
The time-dependent factor in \cref{eq:apx1,eq:apx2,eq:apx3,eq:apx4} is just the memory kernel (or its derivative/integral), and generally does not need to be a power law. The spatial factor, inverse Laplacian (or bi-Laplacian) of the initial data, results from the series expansion in the Laplace domain, and does not depend on specifics of the memory kernel, so long as the memory kernel has the same limit as $t\rightarrow\infty$. This means that the initial-data preservation found in \cref{thm:backwards_problem} should be an expected feature of a broad class of memory influenced models.


\subsection{Comparison with standard diffusion}
\label{subsection:compare_standard_diff}

It is worth comparing the methodology for the inverse problem here (\cref{thm:backwards_problem}) to what happens for the standard heat equation (which is exponentially ill-posed). Using $u(x,t)$ for the solution to Caputo fractional diffusion $\alpha<1$ and $w(x,t)$ for standard diffusion, the differential equations are as follows:
\begin{align}
    D_t w =& \Delta  w & w (x,0) &= v_0(x)
    \\
    \caputo[]{\alpha   }{t}u =& \Delta u  & u(x,0) &= v_0(x).
\end{align}
We can take the solutions at time $t$ to be
\begin{align}
    w(x,t) =& e^{-t(-\Delta)} v_0(x) \label{eq:diff_inverse_p}
    \\
    u(x,t) =& e_{\alpha}\big(t;(-\Delta)\big) v_0(x) \label{eq:subdiff_inverse_p}
\end{align}
(in the sense that  $(-\Delta) \rightarrow k^2$ after the Fourier transform).
Recovering the initial data $v_0$ from \cref{eq:diff_inverse_p,eq:subdiff_inverse_p} for diffusion and subdiffusion, respectively, we have 
\begin{align}
    v_0(x) =& e^{+t (-\Delta)} w(x,t) \label{eq:exp_lap}
    \\
    v_0(x) =&   e_{\alpha,1}\big(t;(-\Delta)\big)  ^{-1} u(x,t).
\end{align}

The operator in \cref{eq:exp_lap}, $e^{+t(-\Delta)}$ exponentially amplifies higher frequencies, making it ill-defined even on most smooth $L_2$ functions. Arbitrarily small noise of high enough frequency can completely overtake the reconstruction even on short timescales, demonstrating how ill-posed the inverse problem for diffusion is.
Although the subdiffusion case does not have the time-translation-invariant property that $e^{t_1}e^{t_2} = e^{t_1+t_2},$ it does not suffer from this exponential-amplification defect. 
In the subdiffusion case, we see through \cref{eq:asymptotic_of_e} that, asymptotically,
\[
e_{\alpha,1}\big(t;(-\Delta)\big) \approx \frac{t^{-\alpha}}{\Gamma(1-\alpha)} (-\Delta)^{-1} - \mathcal O \left(t^{-2\alpha} (-\Delta)^{-2}\right)
\]
so that 
\begin{equation}
     e_{\alpha,1}\big(t;(-\Delta)\big) ^{-1} \approx t^{\alpha}\Gamma(1-\alpha) (-\Delta) \bigg (1- \mathcal O \left (t^{-\alpha} (-\Delta)^{-1}\right ) \bigg)^{-1} \approx t^{\alpha}\Gamma(1-\alpha) (-\Delta)+ \mathcal O(1).
\end{equation}
So for large $t$, $  e_{\alpha,1}\big(t;(-\Delta)\big)^{-1} $ amplifies higher frequencies only linearly in $k^2$, versus exponential amplification for $e^{t(-\Delta)}$. This algebraic amplification in $k$ (quadratic for Caputo, and quartic in the Riemann-Liouville case) compared to the exponential amplification in $k$ for the standard diffusion makes a tremendous difference for the inverse problem with noise, which can be seen in \cref{fig:noisy_reconstruction} and is also touched on in \cite{liuBackwardProblemTimefractional2010}.

\begin{figure}
    \centering
    \includegraphics[width=0.9\linewidth]{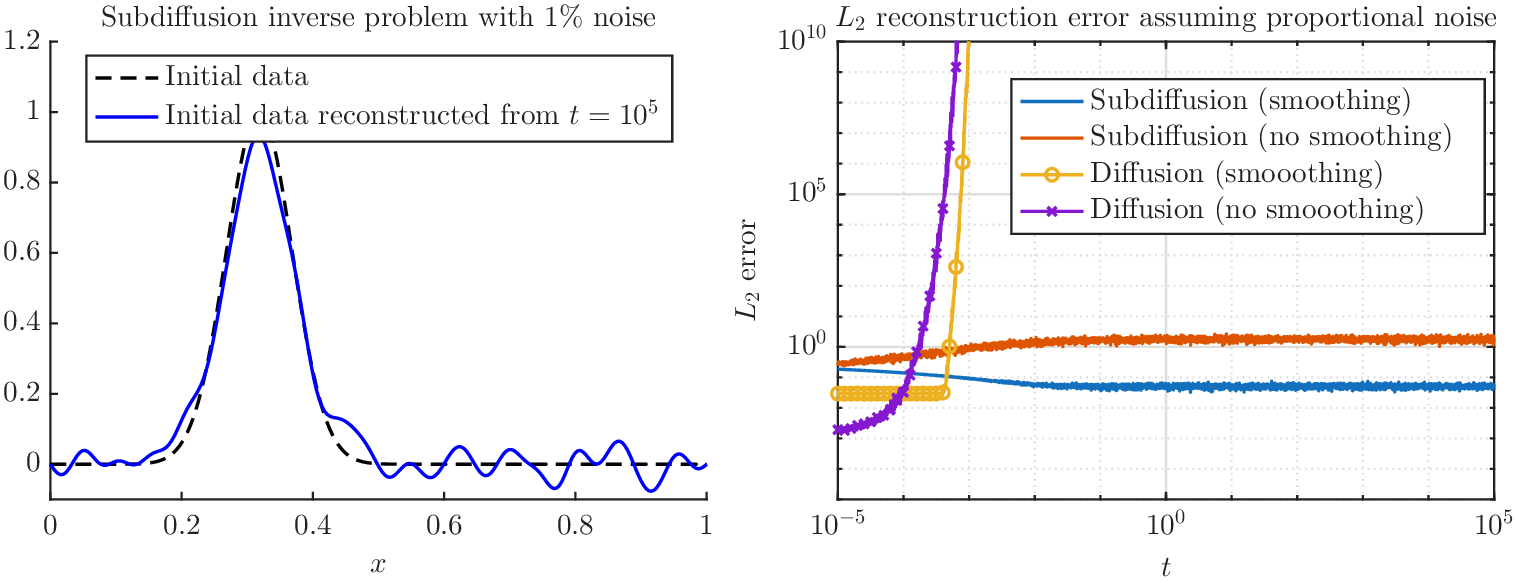}
    \caption{
    Comparison of the stability of the inverse problem for ($\alpha=0.45$ Caputo) anomalous diffusion versus standard diffusion in the presence of $1\%$ measurement noise.  With the same gaussian initial data, at each timestep the solution to the respective equation was calculated, white noise proportional to the $L_2$ norm of the solution at that time was added, and the inverse procedure in \cref{eq:subdiff_inverse_p,eq:diff_inverse_p} was applied following smoothing via gaussian convolution. The plot on the left shows the
    reconstruction results of the initial data for subdiffusion at time $t=10^5$ with $1\%$ noise. %
    The plot on the right shows the $L_2$ error of the reconstructed initial data (with and without gaussian smoothing) for both computed cases. Note that for standard diffusion the reconstruction process diverges rapidly for even small times, while even the \textit{unsmoothed} reconstruction of the subdiffusion solution is relatively well-behaved.
    }
    \label{fig:noisy_reconstruction}
\end{figure}

 \newpage

\section{Characteristic timescales for Mittag-Leffler functions } 
\label{section:Characteristic_time_scales}

 \begin{figure}
     \centering
     \begin{tikzpicture}[scale=.9, decoration={
    markings,
    mark=at position 0.2 with {\arrow{Stealth}},
    mark=at position 0.3 with {\arrow{Stealth}},
    mark=at position 0.4 with {\arrow{Stealth}},
    mark=at position 0.5 with {\arrow{Stealth}},
    mark=at position 0.6 with {\arrow{Stealth}},
    mark=at position 0.7 with {\arrow{Stealth}},
    mark=at position 0.8 with {\arrow{Stealth}}
    mark=at position 0.9 with {\arrow{Stealth}},
    mark=at position 0.99 with {\arrow{Stealth}},}
]

    \def\Rad{3}       
    \def\g{0.07} 
    \def\eps{0.1}   
    \def\ang{15}    

    \draw[->] (-3.5,0) -- (1.5,0) node[right] {Re($s$)};
    \draw[->] (0,-3.5) -- (0,3.5) node[above] {Im($s$)};

    \draw[thick, blue, postaction={decorate}] 
        (\g, -\Rad) -- (\g, \Rad)           
        arc[start angle=90, end angle=180, radius=\Rad -\eps] 
        -- (0, \eps)                             
        arc[start angle=90, end angle=-90, radius=\eps] 
        -- (-\Rad + \eps+\g, -\eps)                            
        arc[start angle=180, end angle=270, radius=\Rad-\eps] 
        -- cycle;

    \draw[red, line width=1.5pt, dashed] (0,0) -- (-3.2,0);

    \filldraw[black] (-0.8, 1.2) circle (1.5pt) node[right] {$s_+$};
    \filldraw[black] (-0.8, -1.2) circle (1.5pt) node[right] {$s_-$};

    \node[blue, right] at (\g, 1.5) {$C_\text{Brom}$};
    \node[blue, above right] at (1.5, 2.5) {$C_R$};
    \node[blue, above left] at (-2,2) {$C_{+i}$};
    \node[blue, below left] at (-2,-2) {$C_{-i}$};
    \node[blue, above] at (-1.5, \eps) {$C_{\text{cut}}$};

\end{tikzpicture}
     \caption{The contour $C_{R}$ used in the proof of \cref{thm:mittag_leff_representation}.
     For $\alpha\in(1,2)$, $\frac{s^{\alpha-\beta}}{s^\alpha+1}$ has poles at $s_{\pm} =  e^{\pm i\frac{\pi}{\alpha}}$, but for $\alpha \in (0,1)$ there are no poles (and so they would not be present in this diagram). The integral on $C_{\pm i}$ goes to zero in the limit as $R\rightarrow \infty$, so the integral along $C_\text{Brom}$ ('Brom' for Bromwich) is equal to pole contributions minus the integral along $C_{\text{cut}}$.
     }
     \label{fig:contour}
 \end{figure}
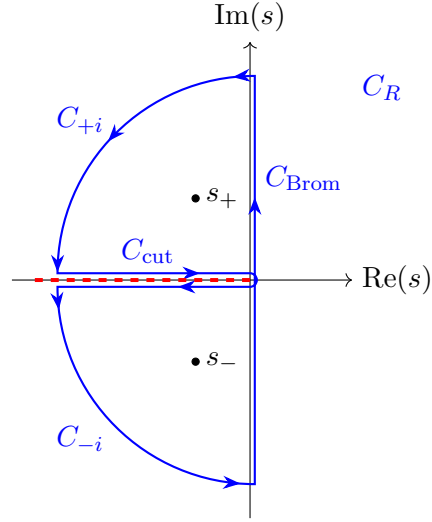

Though we proved convergence as $t\rightarrow\infty$ in the preceding section, we gave no indication at what time scales the approximations \cref{eq:apx1,eq:apx2,eq:apx3,eq:apx4} are accurate. 
We will presently take $\lambda=1$ for simplicity, but because of \cref{eq:relate_e_e_lambda}
\begin{equation*}
    e_{\alpha,\beta}(t;\lambda) = \lambda^{\frac{1-\beta}{\alpha}   } e_{\alpha,\beta}(\lambda^{1/\alpha} t)
\end{equation*}
we will see in \cref{subsection:timescale_and_lambda} that timescales scale with $\lambda^{-1/\alpha}$.

Recall the recurrence relation (\cref{lemma:recurrence_relation}), that
\begin{equation} 
    e_{\alpha,\beta}(t) = \frac{t^{-\alpha+\beta-1} }{ \Gamma(-\alpha+\beta)} -  e_{\alpha,\beta-\alpha}(t).
\end{equation}
From the asymptotic series \cref{eq:asymptotic_of_e}, there must be a timescale $t_*$ after which $ \left| \frac{t^{-\alpha+\beta-1} }{\Gamma(-\alpha+\beta)} \right | \gg \left | e_{\alpha,\beta-\alpha}(t) \right |$.  Such timescales are the focus of this section.
Using an alternate representation,
we will bound components of $e_{\alpha,\beta-\alpha}(t)$, and determine when $\frac{t^{-\alpha+\beta-1} }{\Gamma(-\alpha+\beta)}$ dominates over those bounds. We start with the alternative representation.
\begin{theorem} \label{thm:mittag_leff_representation}
    For $t>0,$ $\alpha\in (0,1)\cup(1,2)$, and $\beta<\alpha-1$ the Mittag-Leffler function can be represented as \begin{equation}
        e_{\alpha,\beta}(t ) = \begin{cases}
            f_{\alpha,\beta}(t) & \alpha \in (0,1)
            \\
            f_{\alpha,\beta}(t) + g_{\alpha,\beta}(t) & \alpha \in (1,2),
        \end{cases}
    \end{equation}
where 
\begin{align}
    f_{\alpha,\beta}(t) =& \frac{-1}{\pi}\int_{0}^\infty e^{-zt} z^{\alpha-\beta}\text{Im}  \frac{e^{i \pi (\alpha-\beta)}}{z^{\alpha  }e^{i \pi \alpha}+1} dz \label{eq:def_f}
    \\
     g_{\alpha,\beta}(t ) =& \frac{2}{\alpha} e^{t  \cos(\pi/\alpha)} \cos
        \left( \sin (\pi/\alpha)  t - (\beta-1) \frac{\pi}{\alpha} \right). \label{eq:def_g}
\end{align}
\end{theorem}
\textbf{Note:} For $\alpha \in (1,2)$  $\cos(\pi/\alpha) $ is negative, leading to exponential decay of $g_{\alpha,\beta}$.

\begin{proof}
  The Laplace transform of $e_{\alpha,\beta}(t )$ is $\frac{s^{\alpha-\beta}}{s^\alpha+1}.$ Using the inverse Laplace transform, we have:
    \begin{align}
       e_{\alpha,\beta}(t ) =  \El^{-1} \left[\frac{s^{\alpha-\beta}}{s^\alpha+1}\right] 
        =& \frac{1}{2 \pi i}\int_{\varepsilon-i\infty}^{\varepsilon+i\infty} e^{st} \frac{s^{\alpha-\beta}}{s^\alpha+1} ds.
        \end{align}
%
Consider the closed contour $C_R$ given by $C_{R} = C_{\text{Brom}} + C_{+i}+ C_{\text{cut}} +C_{-i}$, as displayed in \cref{fig:contour}. Because $C_{\text{Brom}}$ is the contour used in the inverse Laplace transform, we have that
\begin{equation}
    e_{\alpha,\beta}(t )  =  \frac{1}{2 \pi i}  \left ( \int_{C_R } - \int_{C_{\text{cut} } } - \int_{C_{+i} } - \int_{C_{-i} }   \right ) e^{st} \frac{s^{\alpha-\beta}}{s^\alpha+1} ds.
\end{equation}
In the limit as $R\rightarrow\infty$, the integrals on $C_{\pm i}$ go to zero, so we disregard them. The integral over closed contour $C_R$  is given by summing over the enclosed poles. For $\alpha<1,$ $\frac{s^{\alpha-\beta}}{s^\alpha+1}$ has no poles, but for $\alpha\in(1,2),$ it has poles at
\begin{equation}
    s_{\pm} =    e^{\pm i \pi/\alpha} = \cos(\pi/\alpha) \pm i \sin(\pi/\alpha). 
\end{equation}
Hence, for $\alpha \in (1,2)$ summing up the residues gives
\begin{align}
      \frac{1}{2 \pi i}   \int_{C_R } e^{st} \frac{s^{\alpha-\beta}}{s^\alpha+1} ds 
      =& \sum_{{\pm}} e^{ ts_{\pm}}\frac{s_{\pm}^{\alpha-\beta}}{ \alpha s_{\pm}^{\alpha-1}}
       = \sum_{{\pm}} e^{ ts_{\pm}}\frac{s_{\pm}^{-1-\beta}}{ \alpha}
       \\=&\frac{2}{\alpha} e^{t  \cos(\pi/\alpha)} \cos
        \left( \sin (\pi/\alpha)  t - (\beta-1) \frac{\pi}{\alpha} \right),
\end{align}
 which is $g_{\alpha,\beta}(t)$ as stated in \cref{eq:def_g}.
Now we simplify the integral on $C_{\text{cut}}.$
 \begin{align*}
    - \frac{1}{2 \pi i}   \int_{C_{\text{cut}} } e^{st} \frac{s^{\alpha-\beta}}{s^\alpha+1} ds 
    =& \frac{-1}{2\pi i} \left ( \int_{-\infty + 0i}^{0+0i} + \int_{0-0i}^{-\infty-0i}\right) e^{st}\frac{s^{\alpha-\beta}}{s^\alpha+1} ds 
    \\ &\text{taking $s = z e^{\pm i \pi}$, depending on the branch}
       \\ =& \frac{-1}{2\pi i} \int_\infty^0 e^{z e^{i \pi}t} \frac{(z e^{i \pi})^{\alpha-\beta}}{(z e^{i\pi})^{\alpha}+1}  e^{- i \pi} dz + \frac{-1}{2\pi i}\int_0^\infty   e^{z e^{i \pi}t}  \frac{(z e^{-i \pi})^{\alpha-\beta}}{(z e^{-i\pi})^{\alpha}+1}  e^{ i \pi} dz
       \\ =&\frac{ -1}{2\pi i}\int_0^\infty e^{-zt}  z^{\alpha-\beta} \left(   \frac{e^{i \pi(\alpha-\beta)}}{ z^\alpha e^{i \pi \alpha }+1}   -   \frac{e^{-i \pi(\alpha-\beta)}}{ z^\alpha e^{-i \pi \alpha }+1}   \right)dz
       \\
       =& \frac{-1}{\pi} \int_0^\infty e^{-zt} z^{\alpha-\beta} \text{Im}\, \frac{e^{i\pi (\alpha-\beta)}}{z^\alpha e^{i\pi \alpha}+1},
 \end{align*}
 and so \begin{equation}
     e_{\alpha,\beta}(t) = \begin{cases}
         f_{\alpha,\beta}(t) &\alpha\in (0,1)
         \\
         f_{\alpha,\beta}(t) + g_{\alpha,\beta}(t) & \alpha\in (1,2)
     \end{cases}
 \end{equation}
 as claimed.
\end{proof}

Combining this result with the recurrence relation (\cref{lemma:recurrence_relation}), we obtain that
\begin{equation}
    e_{\alpha,\beta}(t) =\begin{cases}
        \frac{t^{-\alpha+\beta-1}}{\Gamma(-\alpha+\beta)} -f_{\alpha,\beta-\alpha}(t ) & \alpha \in (0,1)
        \\
        \frac{t^{-\alpha+\beta-1}}{\Gamma(-\alpha+\beta)} -f_{\alpha,\beta-\alpha}(t ) - g_{\alpha,\beta-\alpha}(t) & \alpha \in (1,2).
    \end{cases} 
\end{equation}
Next, we find the timescales for  $\frac{t^{-\alpha+\beta-1}}{\Gamma(-\alpha+\beta)}$ dominating over $f_{\alpha,\beta-\alpha}(t)$, and dominating over $g_{\alpha,\beta-\alpha}(t).$

\subsection{Characteristic time for domination of $f_{\alpha,\beta-\alpha}(t)$}
We will derive the timescale on which $\frac{t^{-\alpha+\beta-1}}{\Gamma(-\alpha+\beta)}$ dominates over $f_{\alpha,\beta-\alpha}(t)$.

\begin{theorem} \label{thm:char_timescale_f}
    Let $\alpha\in(0,1)\cup(1,2)$, and $\beta <\alpha+1$ such that $\beta-\alpha \not \in \Z_{\leq 0}$, then with 
    \begin{equation}
        \tau(\alpha,\beta) = \begin{cases}
            \left | \frac{ \Gamma(-\alpha + \beta) \Gamma(2\alpha-\beta + 1) }{\pi } \right  | ^{1/\alpha} & \alpha \in(0,1/2)\cup(3/2,2)
            \\
            \left | \frac{ \Gamma(-\alpha + \beta) \Gamma(2\alpha-\beta + 1) }{\pi \sin(\pi\alpha)} \right  | ^{1/\alpha} & \alpha \in (1/2,3/2),
        \end{cases}
    \end{equation}
    we have 
    \begin{equation}
        |f_{\alpha,\beta-\alpha}(t)| <  \varepsilon \left | \frac{t^{-\alpha + \beta-1}}{\Gamma(-\alpha+\beta)} \right | \qquad t> \varepsilon^{-1/\alpha} \tau(\alpha,\beta).
    \end{equation}
\end{theorem}

\begin{proof}
    We will first prove a uniform bound on $f_{\alpha,\hat\beta}(t)$ (taking $\hat \beta= \beta-\alpha$), and then find the time $t_0$ such that $\frac{t^{-\alpha+\beta-1}}{\Gamma(-\alpha+\beta)}$ dominates over that bound for $t\geq t_0$.
    We have
    \begin{align}
      \left |f_{\alpha,\hat\beta}(t) \right | =&  \left | \frac{1}{\pi}\int_{0}^\infty e^{-zt} z^{\alpha- \hat \beta}\text{Im}  \frac{e^{i \pi (\alpha- \hat \beta)}}{z^{\alpha  }e^{i \pi \alpha}+1} dz\right | \notag
      \\
       \leq&  \frac{1}{\pi}\int_{0}^\infty e^{-zt} z^{\alpha-\hat \beta}  \left |\frac{1}{z^{\alpha  }e^{i \pi \alpha}+1}\right |  dz. \label{eq:f_bound_step_2}
    \end{align}
    Now, for $\alpha \in (0,1/2) \cup (3/2,2)$, we have that $\text{Re}[e^{i\pi\alpha}]>0$, and so $| z^\alpha e^{i\pi\alpha}+1| \geq 1 $. For $\alpha\in (1/2,3/2)$, $\text{Re}[e^{i\pi\alpha}]\leq0 $; in this case, $|z^{\alpha}e^{i \pi\alpha}+1| \geq |\sin(\pi\alpha)|$ due to
    \begin{equation}
        |y e^{i\pi\alpha}+1 | = |y + e^{-i\pi\alpha}| \geq |-i\sin(\pi\alpha)| = |\sin( \pi\alpha)|. 
    \end{equation}
   Taking $b_\alpha$ to be either $1$ or $|\sin(\pi\alpha)|$, and developing \cref{eq:f_bound_step_2} we get that
    \begin{equation} \label{eq:bound_on_f}
         \left |f_{\alpha,\hat\beta}(t) \right | \leq \frac{1}{\pi\, b_\alpha} \int_{0}^\infty e^{-zt} z^{\alpha- \hat\beta}  dz = \frac{t^{-\alpha+\hat\beta-1} \Gamma(\alpha-\hat\beta+1)}{\pi \, b_\alpha}.
    \end{equation}
    From this bound we have \begin{equation}
        \left | \frac{f_{\alpha, \beta-\alpha}(t)}{ \left( \frac{t^{-\alpha+\beta-1}}{\Gamma(-\alpha+\beta)} \right)}   \right | \leq   \left | \frac{t^{-\alpha} \Gamma(-\alpha+\beta) \Gamma(2\alpha-\beta+1 )}{\pi \,b_\alpha}  \right |
    \end{equation}
    for all $t>0$.
    Setting the right-hand side equal to $\varepsilon$ and solving for $t$, gives $t_0= \varepsilon^{-1/\alpha} \tau(\alpha,\beta)$, as claimed.
\end{proof}

    \begin{rmk}
    For $\alpha\in(0,1)$, \cref{eq:bound_on_f} is a bound on the Mittag-Leffler function $e_{\alpha,\beta}(t)$, because there is no pole contribution to the integral over $C_R$.
    \end{rmk}

\subsection{Characteristic time for domination of $g_{\alpha,\beta-\alpha}(t)$}
         We will now derive the characteristic time $T(\alpha,\beta),$ after which the term $\frac{t^{-\alpha+\beta-1}}{\Gamma(-\alpha+\beta)}$ dominates over $g_{\alpha,\beta-\alpha}(t)$.

\begin{theorem} \label{thm:char_timescale_g}
    For $1 <\alpha < 2$, $\beta< \alpha+1$, and $\beta - \alpha \not\in \Z_{\leq 0}$,  we have that
    \begin{equation} \label{eq:bound_on_g_aba}
      \left|  g_{\alpha,\beta-\alpha}(t)  \right | \leq   \varepsilon\left|\frac{t^{-\alpha+\beta-1}}{\Gamma(-\alpha+\beta)} \right| \qquad 
      t\geq T_\varepsilon(\alpha,\beta).
    \end{equation}
    where 
    \begin{equation}
         T_\varepsilon(\alpha,\beta) = \frac{\alpha+1-\beta}{\cos(\pi/\alpha)}  W_{-1}\left(  \frac {\cos(\pi/\alpha)}{\alpha+1-\beta}\left(\frac{\alpha}{2}\frac{\varepsilon}{|\Gamma(-\alpha+\beta) |} \right)^{\frac{1}{\alpha+1-\beta}} \right),
    \end{equation}
    ($W_{-1}$ is the $-1$ branch of the Lambert W function).
\end{theorem} 


\begin{proof}
    Recall that \[
    g_{\alpha,\beta}(t) = \frac{2}{\alpha}  e^{t \cos(\pi/\alpha)} \cos \big( t\sin(\pi/\alpha) - ( \beta-1) \pi/\alpha \big).
    \]
    To bound $g$, we can ignore the oscillations. So
    \begin{equation}
        \left|\frac{g_{\alpha,\beta-\alpha }(t ) }{ \left (   \frac{t^{-\alpha+\beta-1}}{ \Gamma(-\alpha+\beta)}  \right)}  \right| \leq \frac{2}{\alpha}e^{t\cos(\pi/\alpha)}  t^{\alpha+1-\beta}  | \Gamma(-\alpha+\beta) |. \label{eq:exp_vs_power_ineq}
    \end{equation}
    Note that with this simplification, the value of $\beta$ in $g_{\alpha,\beta}$ no longer matters, as it only encodes phase information for the oscillations.
  As $\cos(\pi/\alpha) <0$, the exponential factor $\exp( t \cos(\pi/\alpha))$ will eventually decay faster than $t^{\alpha+1-\beta}$ grows, pushing the right-hand side of \cref{eq:exp_vs_power_ineq} to arbitrarily small values. Labeling the time the right-hand side of \cref{eq:exp_vs_power_ineq} reaches $\varepsilon$ as $ t =  T_\varepsilon$, we can solve
   \begin{align}
      \frac{2}{\alpha} e^{ \ T_\varepsilon \cos(\pi/\alpha)} T_\varepsilon^{\alpha+1-\beta}  | \Gamma(-\alpha+\beta) | =& \varepsilon \label{eq:when_it_crosses_def_T}
       \\
       e^{T_\varepsilon \cos(\pi/\alpha)} T_\varepsilon^{\alpha+1-\beta} =&  \frac{\alpha}{2}\frac{\varepsilon}{|\Gamma(-\alpha+\beta) |} \notag
       \\
       e^{T_\varepsilon \frac {\cos(\pi/\alpha)}{\alpha+1-\beta}} \frac {\cos(\pi/\alpha)}{\alpha+1-\beta}T_\varepsilon =&  \frac {\cos(\pi/\alpha)}{\alpha+1-\beta}\left(\frac{\alpha}{2}\frac{\varepsilon}{|\Gamma(-\alpha+\beta) |} \right)^{\frac{1}{\alpha+1-\beta}}. \notag
   \end{align}
This is a transcendental equation of the form \(y e^y = z,\)
which, for $z\in [-1/e,0)$ has real solutions at 
 $y = W_{-1}(z),W_{0}(z)$, where $W_{}$ is the Lambert W function. Because we want the latest time that \cref{eq:when_it_crosses_def_T} is satisfied, and $ W_{-1}(z) \leq W_{0}(z) < 0$, we take the $-1$ branch of the Lambert W function. Thus
\begin{align*}
    \frac {\cos(\pi/\alpha)}{\alpha+1-\beta}T_\varepsilon  =& W_{-1} \left(  \frac {\cos(\pi/\alpha)}{\alpha+1-\beta}\left(\frac{\alpha}{2}\frac{\varepsilon}{|\Gamma(-\alpha+\beta) |} \right)^{\frac{1}{\alpha+1-\beta}} \right)
    \\
    T_\varepsilon =& \frac{\alpha+1-\beta}{\cos(\pi/\alpha)}  W_{-1}\left(  \frac {\cos(\pi/\alpha)}{\alpha+1-\beta}\left(\frac{\alpha}{2}\frac{\varepsilon}{|\Gamma(-\alpha+\beta) |} \right)^{\frac{1}{\alpha+1-\beta}} \right),
\end{align*}
    which is $T_{\varepsilon}(\alpha,\beta)$.
\end{proof} 
\begin{rmk}
     Because we only removed phase information, $T_\varepsilon(\alpha,\beta)$ is a much more precise timescale than $\varepsilon^{-1/\alpha}\tau(\alpha,\beta)$, where bounding the integral in $f_{\alpha,\beta}(t)$ with just a  power law loses much more information. To notionally match $\tau(\alpha,\beta)$, we will take $T(\alpha,\beta) =T_\varepsilon(\alpha,\beta)|_{\varepsilon=1}$ when we consider them as general 'characteristic' timescales.
\end{rmk}

 \subsection{Applying timescales to Mittag-Leffler functions}
 In the preceding subsections, we found characteristic times $\tau(\alpha,\beta)$ and $T(\alpha,\beta)$, 
 \begin{align}
     \tau(\alpha,\beta) =&
     \begin{cases}  \left | \frac{ \Gamma(-\alpha + \beta)  \Gamma(2\alpha-\beta + 1) }{\pi} \right  | ^{1/\alpha} &\alpha \in (0,1/2)\cup(3/2,2)
     \\
      \left | \frac{ \Gamma(-\alpha + \beta) \Gamma(2\alpha-\beta + 1) }{\pi \sin(\pi\alpha)} \right  | ^{1/\alpha} &\alpha \in (1/2,1)\cup(1,3/2)
     \end{cases}
     \\T(\alpha,\beta)  =& \frac{\alpha+1-\beta}{\cos(\pi/\alpha)}  W_{-1}\left(  \frac {\cos(\pi/\alpha)}{\alpha+1-\beta}\left(\frac{\alpha}{2}\frac{1}{|\Gamma(-\alpha+\beta) |} \right)^{\frac{1}{\alpha+1-\beta}} \right) \qquad \alpha \in (1,2).
 \end{align}

\begin{figure}
    \centering
    \includegraphics[width=0.9\linewidth]{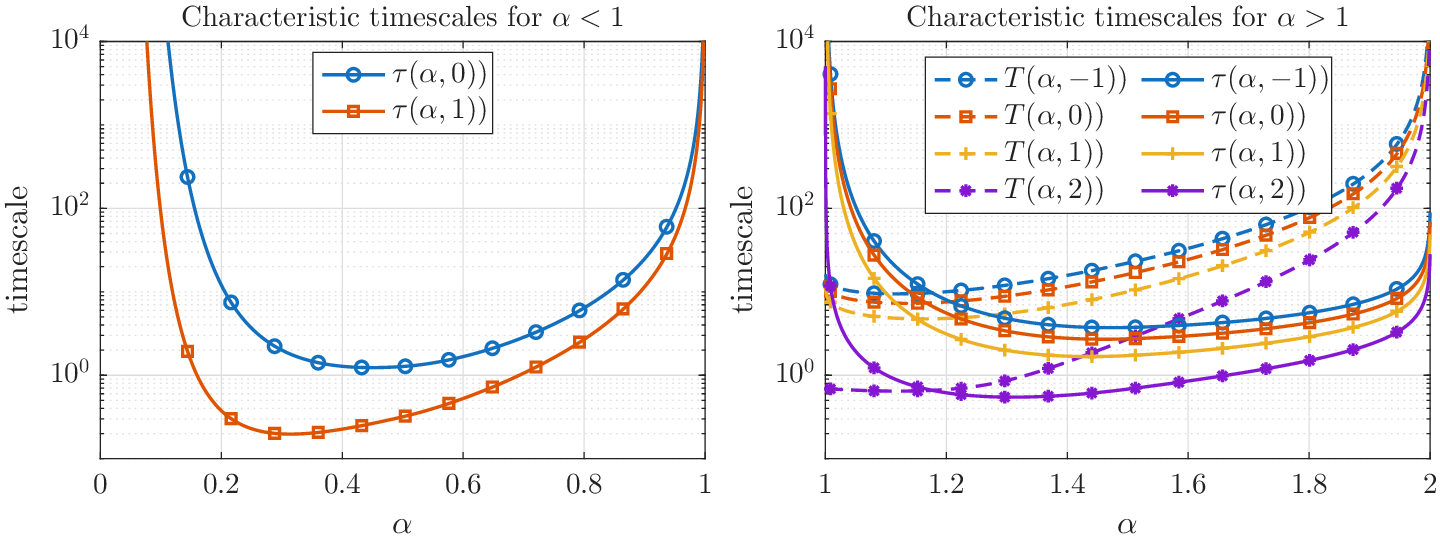}
    
    \caption{
    The characteristic timescales of the Mittag-Leffler function are plotted for $\alpha \in (0,1)$ on the left and $\alpha\in(1,2)$ on the right. 
    $\tau(\alpha,\beta)$ is the characteristic timescale for $\frac{t^{-\alpha+\beta-1}}{\Gamma(-\alpha+\beta)}$ dominating over $f_{\alpha,\beta-\alpha}(t)$ (see \cref{thm:char_timescale_f}), and $T(\alpha,\beta)$ is the timescale for $\frac{t^{-\alpha+\beta-1}}{\Gamma(-\alpha+\beta)}$ dominating over over the oscillatory contribution to the Mittag-Leffler function, $g_{\alpha,\beta}(t)$ (\cref{thm:char_timescale_g}). While there are only two for $\alpha<1$ (corresponding to Caputo and Riemann-Liouville cases), there are eight on the right, due to two types of initial data for both Caputo and Riemann-Liouville, and the two competing timescales in each of those four cases.
    }
    
    \label{fig:tau_T_comparisons}
\end{figure}
Recall from \cref{thm:spectral_rep_u_xi} that we have four Mittag-Leffler functions under consideration, 
$e_{\alpha,1}(t),e_{\alpha,2}(t)$ for the Caputo case, and $e_{\alpha,\alpha}(t),e_{\alpha,\alpha-1}(t)$ for the Riemann-Liouville case. 
Applying these timescales to the Caputo case is fairly straightforward:
\begin{equation}\label{eq:e_a_ell_decomposition}
    e_{\alpha,\ell}(t) = \begin{cases}
        \frac{t^{-\alpha+\ell-1}}{\Gamma(-\alpha+\ell)} - f_{\alpha, \ell-\alpha}(t) & \alpha<1
        \\
         \frac{t^{-\alpha+\ell-1}}{\Gamma(-\alpha+\ell)} - f_{\alpha,\ell-\alpha}(t) - g_{\alpha,\ell-\alpha}(t) & \alpha>1
    \end{cases}
\end{equation}
for $\ell = 1,2$. The term $\frac{t^{-\alpha+\ell-1}}{\Gamma(-\alpha+\ell)}$ starts to dominate over $f_{\alpha,\ell-\alpha}(t)$ around time $\tau(\alpha,\ell)$ (\cref{thm:char_timescale_f}), and $g_{\alpha,\ell-\alpha}(t)$ around time $T(\alpha,\ell)$ (\cref{thm:char_timescale_g}).

The parameter $\beta$ for the Riemann-Liouville modes takes a value such that $\beta-\alpha \in \Z_{\leq0}$, which violates the conditions of \cref{thm:char_timescale_f,thm:char_timescale_g} because $\frac{t^{\alpha- (\alpha - j)-1}}{\Gamma(\alpha-(\alpha -j))} \equiv 0$ and will not dominate anything. To avoid this issue, we apply the recurrence relation once more:
\begin{equation}
    e_{\alpha,\alpha-j}(t) =0-e_{\alpha,-j}(t) =  - \frac{t^{-\alpha-j-1}}{\Gamma(-\alpha-j)} + e_{\alpha,-j-\alpha}(t),
\end{equation}
and so we represent $e_{\alpha,\alpha-j}(t)$ 
\begin{equation}\label{eq:e_a_j_decomposition}
    e_{\alpha,\alpha-j}(t) = - e_{\alpha,-j}(t) = 
    \begin{cases}
        - \frac{t^{-\alpha-j-1}}{\Gamma(-\alpha-j)} + f_{\alpha,-\alpha-j}(t) & \alpha<1
        \\        
        - \frac{t^{-\alpha-j-1}}{\Gamma(-\alpha-j)} + f_{\alpha,-\alpha-j}(t) + g_{\alpha,-\alpha-j}(t) & \alpha>1.
    \end{cases}
\end{equation}
Thus, $\tau(\alpha,-j)$ and $T(\alpha,-j)$ are the characteristic times of $e_{\alpha,\alpha-j}(t).$

All relevant timescales, $\tau(\alpha,\beta),T(\alpha,\beta)$ for $\beta \in\{ 2,1,0,-1\}$ are plotted in their respective domains in \cref{fig:tau_T_comparisons}. Therein note that for $\alpha \gtrsim1.2,$ we have $T(\alpha,\beta) > \tau(\alpha,\beta)$ for all four cases of $\beta$. In this case, when  $\frac{t^{-\alpha+\beta-1}}{\Gamma(-\alpha+\beta)}$ equals the oscillatory pole contribution $g_{\alpha,\beta-\alpha}(t)$  in magnitude (at time $t=T(\alpha,\beta)$) the other contribution $f_{\alpha,\beta-\alpha}(t)$ is largely irrelevant in comparison. 

Surprisingly, this allows us to accurately approximate the number of real zeros of the Mittag-Leffler function $e_{\alpha,\beta}(t)$, this is discussed in \cref{section:zereos_of_ml}.

\subsection{Timescales for different modes and lengthscales}
\label{subsection:timescale_and_lambda}

The Fourier modes for the anomalous diffusion equations, given in \cref{thm:spectral_rep_u_xi}, are $e_{\alpha,\gamma+1}(t;k^2).$ Recall the relation between $e_{\alpha,\beta}(t;\lambda)$ from \cref{eq:relate_e_e_lambda}, 
\begin{equation}
    e_{\alpha,\beta}(t;\lambda) = \lambda^{\frac{1-\beta}{\alpha}}e_{\alpha,\beta}( \lambda ^{1/\alpha} t).
\end{equation}
Because we are comparing between competing terms of the Mittag-Leffler function, the  overall prefactor $\lambda^{\frac{1-\beta}{\alpha}}$ is irrelevant, and we can conclude that the important timescales scale like $ \lambda^{1/\alpha}t = t_{*}$, where $t_*$ is the relevant timescale for $e_{\alpha,\beta}(t)$. So for the Caputo case,  (similar to \cref{eq:e_a_ell_decomposition})
\begin{equation}
    e_{\alpha,\ell}(t;\lambda) = \lambda^{\frac{1-\ell}{\alpha}} e_{\alpha,\ell}(\lambda^{1/\alpha}t) = \lambda^{\frac{1-\ell}{\alpha}}
    \begin{cases}
        \frac{(\lambda^{1/\alpha}t)^{-\alpha+\ell-1}}{ \Gamma(-\alpha+\ell)} - f_{\alpha,\ell-\alpha}(\lambda^{1/\alpha}t) &\alpha<1
        \\
         \frac{(\lambda^{1/\alpha}t)^{-\alpha+\ell-1}}{ \Gamma(-\alpha+\ell)} - f_{\alpha,\ell-\alpha}(\lambda^{1/\alpha}t) - g_{\alpha,\ell-\alpha}(\lambda^{1/\alpha}t) &\alpha>1
    \end{cases}
\end{equation}
 the asymptotic term starts to dominate over $f_{\alpha,\ell-\alpha}$ around time $\lambda^{-1/\alpha}\tau(\alpha,\ell)$, and over $g_{\alpha,\ell-\alpha}$ around time $\lambda^{-1/\alpha}T(\alpha,\ell)$.
Likewise, for the Riemann-Liouville, (similar to \cref{eq:e_a_j_decomposition})
\begin{equation}
    e_{\alpha,\alpha-j}(t;\lambda)
    = \lambda^{\frac{1+j-\alpha}{\alpha}} e_{\alpha,\alpha-j}(\lambda^{1/\alpha}t) 
    = \lambda^{\frac{1+j-\alpha}{\alpha}}
    \begin{cases}
        - \frac{(\lambda^{1/\alpha}t)^{-\alpha-j-1}}{\Gamma(-\alpha-j)} + f_{\alpha,-\alpha-j}(t\lambda^{1/\alpha}) & \alpha<1
        \\        
        - \frac{(\lambda^{1/\alpha}t)^{-\alpha-j-1}}{\Gamma(-\alpha-j)} + f_{\alpha,-\alpha-j}(\lambda^{1/\alpha}t) + g_{\alpha,-\alpha-j}(\lambda^{1/\alpha}t) & \alpha>1
    \end{cases}
\end{equation}
with the asymptotic term dominating over $f_{\alpha,-\alpha-j}$ around $t=\lambda^{-1/\alpha} \tau(\alpha,-j)$, and dominating over $g_{\alpha,-\alpha-j}$ around $t=\lambda^{-1/\alpha} T(\alpha,-j)$.

For anomalous diffusion we consider $\lambda= k^2$, so that timescales for the $k$th mode of $u_{\gamma}$ scale like $t_{*,k} = (k^2)^{-\alpha} t_{*}$. As frequency $k$ increases, the timescale decreases, meaning that short-wavelength features reach their asymptotic form  much more quickly than long-wavelength features. 
Generally,  $k^2 \propto (\text{length})^{-2}$, so that features of length $L$ correspond to timescales of $L^{2/\alpha} t_*$. This relation between time and length is clearest when reconstructing the initial data using the Laplacian or bi-Laplacian through \cref{eq:init_data_recov_0,eq:init_data_recov_1,eq:init_data_recov_a1,eq:init_data_recov_a2}. Therein, features in the initial data of length $L$ resolve in the reconstruction at around time 
\begin{equation}
    t_{\text{res}} \approx \begin{cases} L^{2/\alpha} 
         \tau(\alpha,\beta) &\alpha <1.2
         \\
          L^{2/\alpha}  T(\alpha,\beta)  & \alpha>1.2
    \end{cases}
    \label{eq:relate_length_characteristic_times}
\end{equation}
(the transition at $\alpha=1.2$ is somewhat loose, coming from  $T(\alpha,\beta) > \tau(\alpha,\beta)$ for $\alpha>1.2$).
To demonstrate this, \cref{fig:length_features_subdiff,fig:length_features_superdiff} show the solution (and initial data reconstruction) of the $\gamma=0$ Caputo problem for $\alpha=0.75$ and $\alpha=1.3$ at the characteristic times corresponding to the widths of gaussians in the initial data.






\begin{figure}
    \centering
    \includegraphics[width=0.95\linewidth]{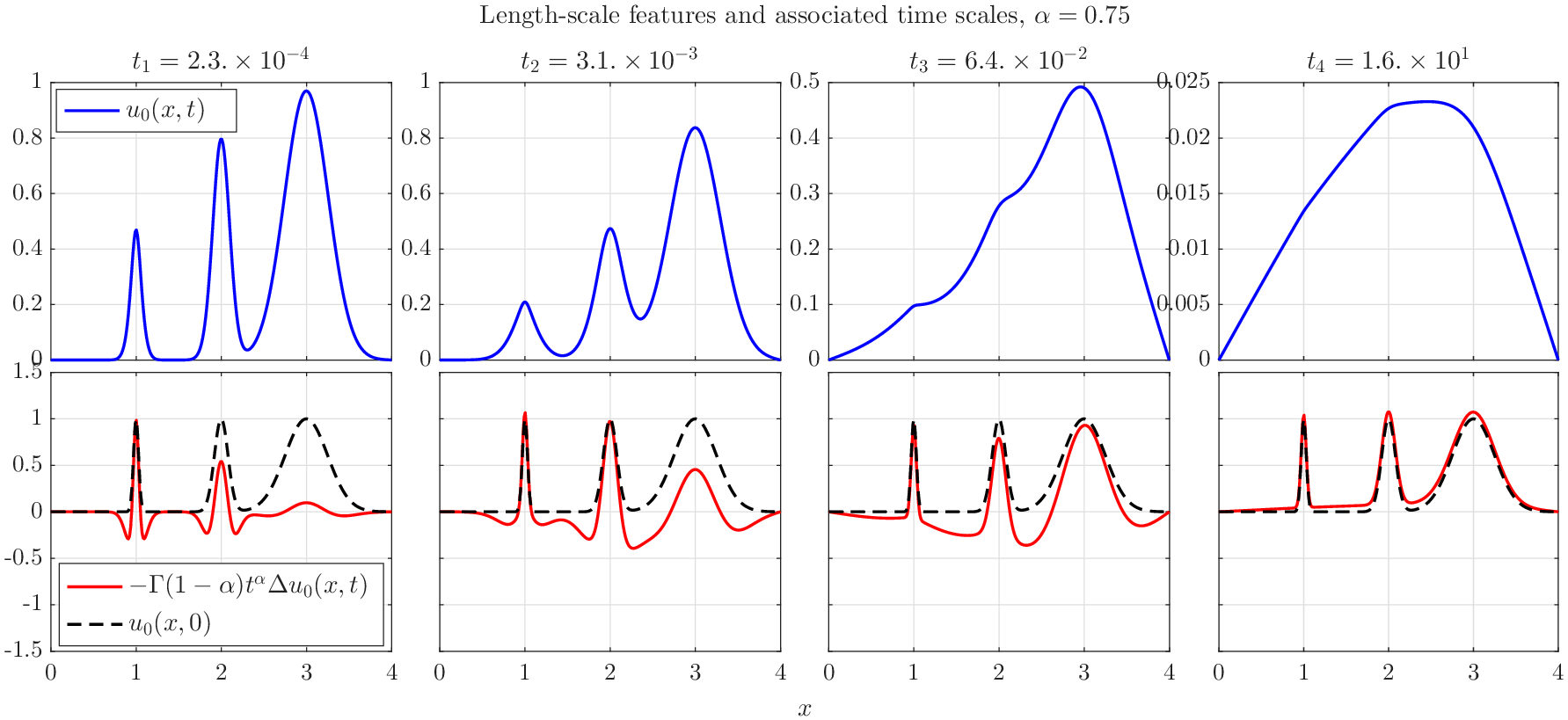}
    \caption{
    The plots display the solution and initial data reconstruction for Caputo ($\gamma=0$ in \cref{def:anom_diff_u_xi}) subdiffusion, $\alpha=0.75$. The initial data is a superposition of three unnormalized gaussians with standard deviations of $\sigma_1 = 0.03$, $\sigma_2 = 0.08$, and $\sigma_3 = 0.25$. The upper row shows the function values $u_0(x,t_j )$ for times $t_j$, $j=1,\cdots,4$. The lower row shows the initial data $u_0(x,0)$ plotted together with $-\Gamma(1-\alpha) t_j^\alpha \Delta u_0(x,t_j)$, the reconstruction described in \cref{thm:backwards_problem}, at each time $t_j$. The times $t_j$ correspond to the length scales in the problem (see \cref{eq:relate_length_characteristic_times}), $t_j = (\sigma_j)^{2/\alpha} \tau(\alpha,1)$, for $j=1,2,3,$ and $t_4 = (4)^{2/\alpha} \tau(\alpha,1) $ ($4$ being the length of the domain).
    Notice how each gaussian in the initial data reconstruction is approximately resolved at its associated timescale. 
    }
    \label{fig:length_features_subdiff}
\end{figure}

\begin{figure}
    \centering
    \includegraphics[width=0.95\linewidth]{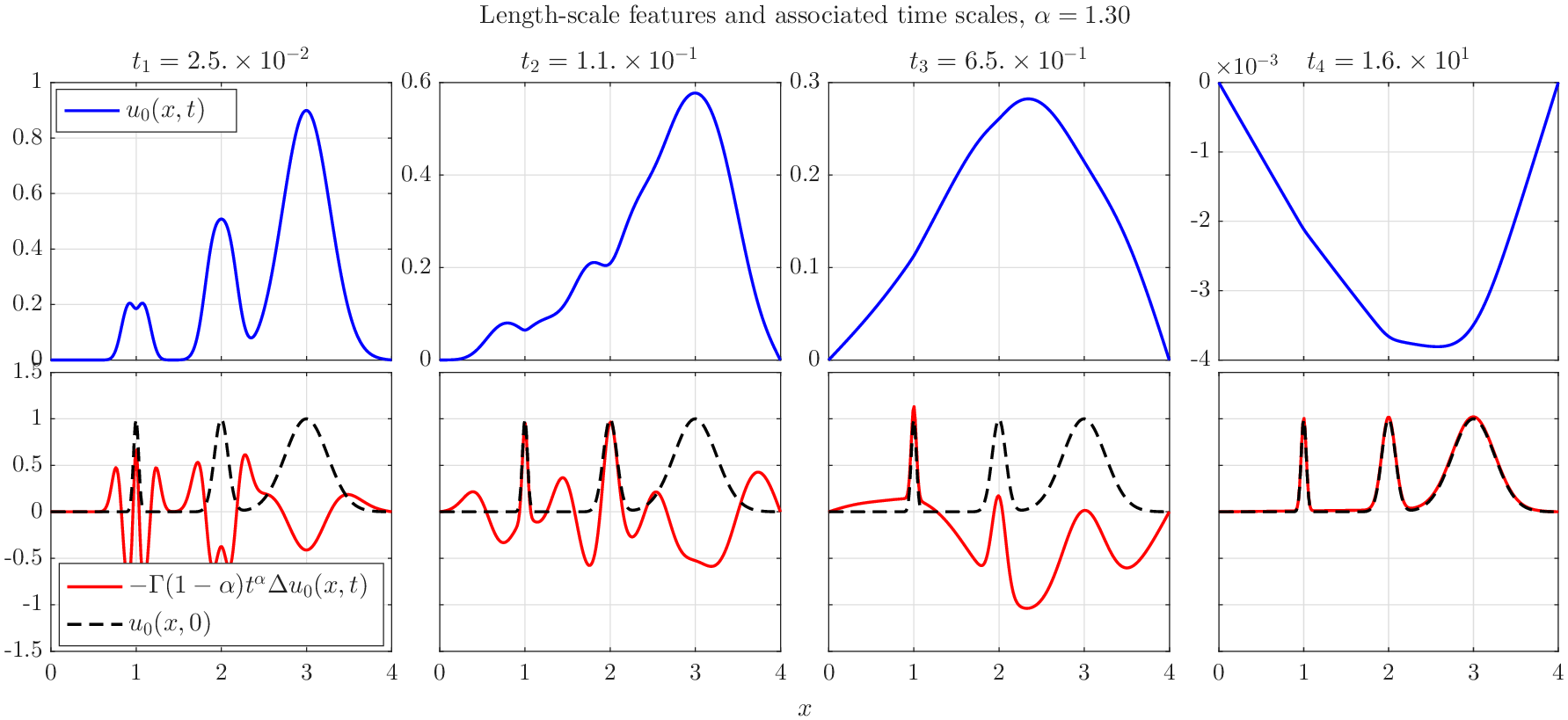}
    \caption{
    This displays the same setup as in \cref{fig:length_features_subdiff}, except with $\alpha=1.3$ -- corresponding to superdiffusion. Here,  $t_j = (\sigma_j)^{2/\alpha} T(\alpha,1)$, for $j=1,2,3,$ and $t_4 = (4)^{2/\alpha} T(\alpha,1) $.
    Note that superdiffusion solutions behave much more wave-like than subdiffusion or diffusion solutions, which is seen above.
    }
    \label{fig:length_features_superdiff}
\end{figure}

\section{Real zeros of the Mittag-Leffler function}

\label{section:zereos_of_ml}

We can construct an approximation for the number of zeros the Mittag-Leffler function $e_{\alpha,\beta}$ has for $t>0$ (as a function of $\alpha$ and $\beta$) using  the results of \cref{section:Characteristic_time_scales}. The problem of enumerating the real zeros of the Mittag-Leffler function is considered in \cite{hannekenEnumerationRealZeros2007,eloeExistenceComparisonMonotonicity2026,aleroevFractionalSturmLiouvilleProblem2025,abooaliComprehensiveStudyZeros2024}.

Though the zeros of Mittag-Leffler functions are not as regular as those of sinusoids, they are not sporadic.  Plotting (for a set value of $\beta$) the $e_{\alpha,\beta}(t) = 0$ contour in the $t \times \alpha$ plane, one can see a clear pattern (see \cref{fig:ML_contours}). The characteristic time $T_{\alpha,\beta}$ (as defined in \cref{section:Characteristic_time_scales}) separates the $t\times \alpha$ plane into a region with (mostly) periodic zeros, and a region with no zeros. This is to say that, for given values of $\alpha$ and $\beta$, $e_{\alpha,\beta}(t)$ has fairly regularly spaced zeros until around $T(\alpha,\beta)$. 
\begin{figure}
    \centering
    \includegraphics[width=0.9\linewidth]{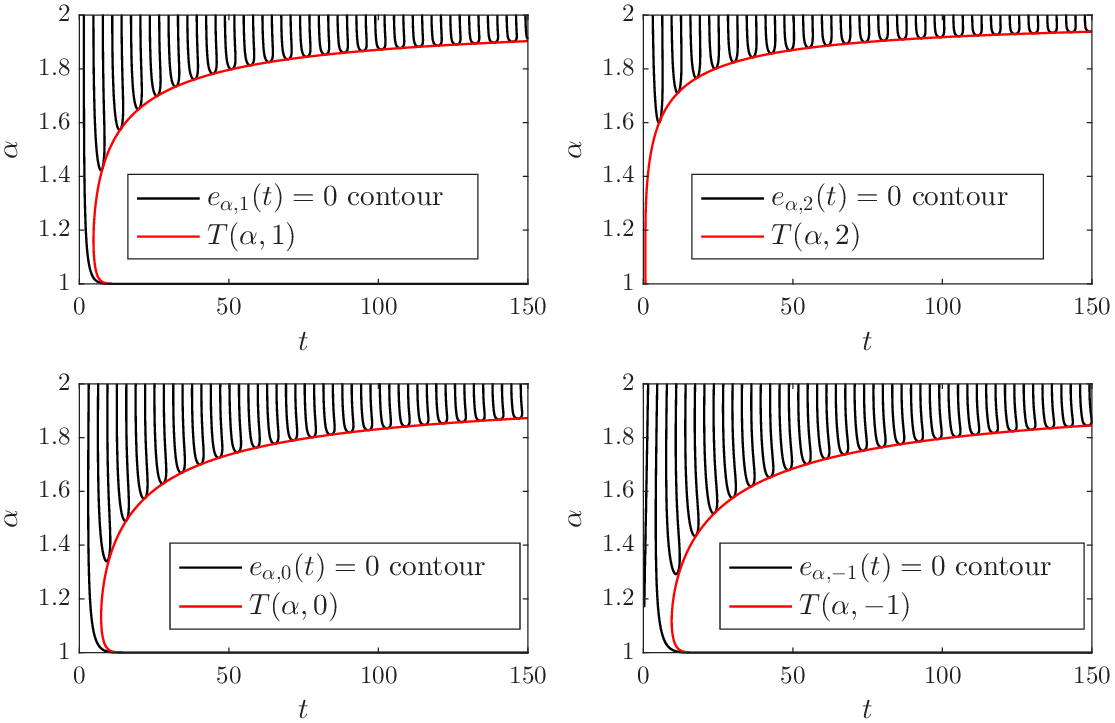}
    \caption{For various values of $\beta,$ the zero contour of the Mittag-Leffler function $e_{\alpha,\beta}(t)= 0$ plotted on the $t\times\alpha$ plane. The timescale $T(\alpha,\beta)$ is plotted as a function of $\alpha$. Note that $T(\alpha,\beta)$  breaks the $t\times \alpha $ plane into a region with zeros $t<T(\alpha,\beta)$ and a region with no zeros $t>T(\alpha,\beta)$.}
    \label{fig:ML_contours}
\end{figure}
To understand this behavior more fully, we can return to the analysis of \cref{section:Characteristic_time_scales}, that for $\alpha\in(1,2),$
\begin{equation}
    e_{\alpha,\beta}(t) = f_{\alpha,\beta}(t) + g_{\alpha,\beta}(t)
\end{equation}
with
\begin{align}
    &f_{\alpha,\beta}(t) \approx \frac{t^{-\alpha+\beta}}{\Gamma(-\alpha+\beta)} \qquad (\text{after $t \approx \tau(\alpha,\beta) $})
    \\
    &g_{\alpha,\beta}(t) = \frac{2}{\alpha} e^{ t \cos(\pi / \alpha)} \cos \left( \sin(\pi/\alpha) t - (1-\beta) \frac{\pi}{\alpha} \right).
\end{align}
$T(\alpha,\beta)$, defined as 
\begin{equation}
T(\alpha,\beta) =\frac{\alpha+1-\beta}{\cos(\pi/\alpha)}  W_{-1}\left(  \frac {\cos(\pi/\alpha)}{\alpha+1-\beta}\left(\frac{\alpha}{2}\frac{1}{|\Gamma(-\alpha+\beta) |} \right)^{\frac{1}{\alpha+1-\beta}} \right) 
\end{equation} 
 is precisely the time at which the exponential decay $\frac{2}{\alpha} e^{\cos(\pi/\alpha) t}$ falls below the long-tailed term $\frac{t^{-\alpha+\beta-1}}{\Gamma(-\alpha+\beta)}$ (\cref{thm:char_timescale_g}). So until (approximately) $t=T(\alpha,\beta)$ the oscillations (with $\omega = \sin(\pi/\alpha) $) in $g_{\alpha,\beta}(t)$ have high enough amplitude to cancel out $f_{\alpha,\beta}(t)$ giving $e_{\alpha,\beta}(t)$ zeros. 
Speaking loosely, the number of zeros that $e_{\alpha,\beta}(t)$ has should be twice the number of oscillations completed by $g_{\alpha,\beta}$, which means that
\begin{equation}
n_{\text{zeros}}   \approx 2 \left( \frac{  \sin(\pi/\alpha)\,T(\alpha,\beta)}{2\pi} \right).
\end{equation}

This estimate can be further refined by analyzing the phase of the oscillations and whether $e_{\alpha,\beta}(t)$ crosses $0$ an even or odd number of times. For example, in the  case of $\beta=1:$ $e_{\alpha,1}(0) = 1>0,$ and $e_{\alpha,1}(t\rightarrow\infty) = \frac{t^{-\alpha}}{\Gamma(-\alpha+1)} <0$ (due to the sign of $\Gamma(-\alpha+1) $ for $\alpha\in (1,2)$), so there must be an additional crossing of the axis in addition to the $2\times\text{(\# completed periods)}$.  
This gives the estimate of  $2\left \lfloor \frac{\sin(\pi/\alpha) T(\alpha,\beta)}{2\pi}\right \rfloor+1$ (where $\lfloor \cdot \rfloor$ denotes the floor function). This approximation is shown in \cref{fig:ML_zero_counting}, plotted together with the numerically counted number of zeros of $e_{\alpha,\beta}(t)$  (numerically calculating the Mittag-Leffler function using \cite{garrappaNumericalEvaluationTwo2015} in MATLAB).

\begin{figure}
    \centering
    \includegraphics[width=0.7\linewidth]{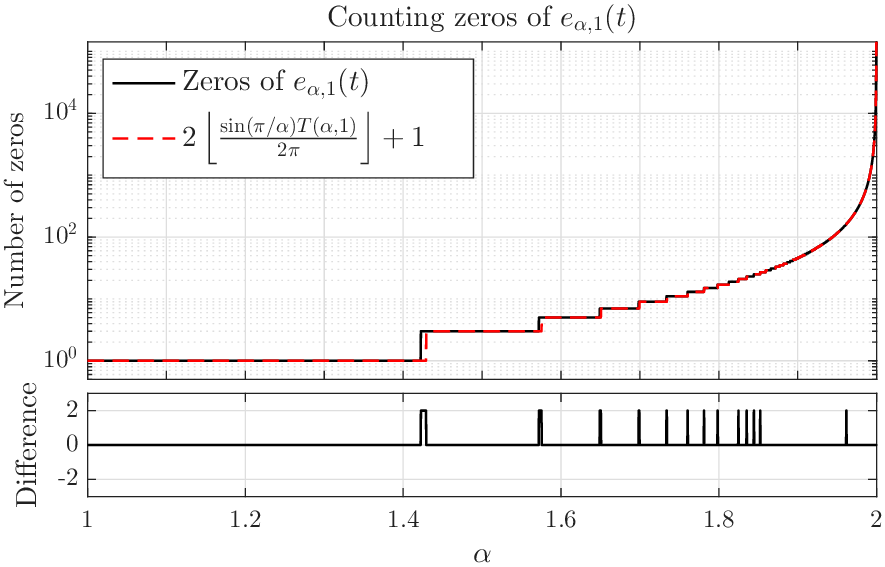}
    \caption{
    The top panel shows the number of zeros of the Mittag-Leffler function $e_{
    \alpha,\beta}(t)$ for $t\in\R$ as a function of $\alpha$ (computed for 2500 evenly spaced values of $\alpha\in [1.0001,1.9999]$ in MATLAB), compared with the estimate $2\left \lfloor \frac{\sin(\pi/\alpha) T(\alpha,\beta)}{2\pi}\right \rfloor+1$. The bottom panel shows the difference between the counted number of zeros  and the estimated number of zeros, demonstrating the accuracy even as the number of zeros diverges.}
    \label{fig:ML_zero_counting}
\end{figure}


\section{Transport on a comb} \label{section:comb_diffusion}
Heat transfer on a comb is an analytically solvable physical model capturing transport in fractal media. Subdiffusion of both Caputo and Riemann-Liouville types  (with $\alpha=1/2$) describes aspects of the transport on a comb \cite{iominFractionalDynamicsComblike2018,arkhincheevAnomalousDiffusionDrift1991,sandevHeterogeneousDiffusionComb2018}. The comb model has been used in many applications, including the motion along the spiny dendrites of nerve cells \cite{mendezComblikeModelsTransport2013} or fluid flow in a porous medium. 
The comb model emphasizes 'traps' which block transport, a feature known to cause anomalous diffusion in more general porous media \cite{arkhincheevAnomalousDiffusionCharge2000}.
The 'comb' consists of a 'spine' along the $x$-axis with 'ribs' extending along $y$ attached at every point on the spine (see \cref{fig:comb_figure}). Heat transport occurs only along the spine or a rib; there is no transfer between ribs except at the spine.

%
 

\begin{figure}
    \centering
    \includegraphics[width=0.35\linewidth]{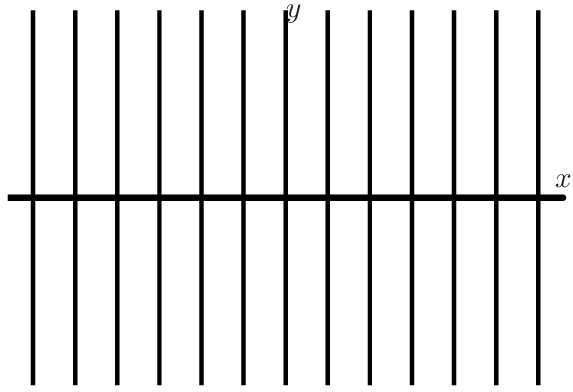}
    \caption{A rough depiction of the 'comb' medium. 
    A single 'spine' along the $x$-axis connects a continuum of otherwise disconnected 'ribs' along the $y$ direction. Note that the ribs are attached to every point of the spine and extend infinitely away from it, though here they are depicted as finite and discrete in spacing.
    }
    \label{fig:comb_figure}
\end{figure}

Let $h(x,y,t)$ be the temperature in the domain $[0,1]\times \R\times\R_{\geq0} $.
We assume that the initial heat distribution is restricted to the spine and that Dirichlet conditions hold at $ x=0$ and $ x=1$. 
Although the comb is topologically distinct from the plane, we model diffusion on the comb by anisotropic diffusion in the plane using the diffusion tensor $D$.
The problem governing heat transport is the following:
\begin{equation} 
   \left \{ \begin{aligned}
        &\partial_t h - \nabla \cdot D \nabla h = 0 
         \qquad \text{for } (x,y,t) \in  [0,1]\times \R\times\R_{\geq0}
        \\
        &h(x,y,0) = \delta(y) v(x)      ,
    \end{aligned} \right.
\end{equation}
Here, the anisotropic thermal conductivity (or anisotropic diffusion) tensor $D$ is: 
\begin{equation}
    D = \begin{pmatrix}
        c_x \delta(y) & 0
        \\
        0 & c_y
    \end{pmatrix},
\end{equation}
where $c_x$, $c_y$ are the thermal conductivities of the spine and ribs, respectively.
In the resulting diffusion equation, $\partial_x^2$ is multiplied by a delta function in $y$, ensuring that transfer along $x$ is only possible at $y=0$: 
\begin{equation} \label{eq:comb_diffusion_setup}
   \left \{ \begin{aligned}
        &\partial_t h -  \delta(y)c_x \partial_{x}^2 h - c_y \partial_y^2 h = 0 \qquad \text{for } (x,y,t) \in  [0,1]\times \R\times\R_{\geq0}
        \\
        &h(x,y,0) = \delta(y) v(x)      ,
    \end{aligned} \right.
\end{equation}
%
Taking the Laplace transform with respect to $t$ and the Fourier transform in $x$ (so $h(x,y,t) \rightarrow \tilde H(k,y,s)$), we reduce \cref{eq:comb_diffusion_setup} to
\begin{equation}
    s \tilde H(k,y,s) - \tilde v(k)\delta(y) + c_x k^2 \delta(y) \tilde H(k,y,s) - c_y \partial_y^2 \tilde H(k,y,s) = 0.
\end{equation}
This is a differential equation in $y$, with solution given as \cite{arkhincheevAnomalousDiffusionCharge2000} 
\begin{equation} \label{eq:comb_sol_LF}
    \tilde H(k,y,s) = \tilde v(k) \frac{ \exp \left(- \sqrt{s/c_y} |y| \right)}{2 \sqrt{sc_y} + c_x k^2}.
\end{equation}
%
    Considering heat transfer in the $x$ direction, we restrict $y$ to $y=0$, 
    \begin{equation}
         \tilde H(k,0,s) = \frac{\tilde v(k)}{2 \sqrt{c_y}}  \frac{1}{s^{1/2} +\frac{c_x}{2 \sqrt{c_y}}k^2 } 
         \, \, \xrightarrow{\El^{-1}}  \, \,
         \frac{\tilde v(k)}{2 \sqrt{c_y}} e_{1/2,1/2}(t;k^2) \label{eq:H_restricted},
    \end{equation}
    and obtain a solution to the Riemann-Liouville subdiffusion problem with $\alpha=1/2:$
    \begin{equation}
  \left \{  \begin{aligned}
    &\RL[]{1/2}{t} u_{\text{r}}- \frac{c_x}{2\sqrt{c_y}} \partial_x^2 u_{\text{r}} = 0
    \\
     &\RL[]{-1/2}{t}[u_{\text{r}}](x,t \rightarrow 0_+)  = \frac{1}{2 \sqrt {c_y}} v(x) ,
    \end{aligned}  \right.
\end{equation}
Here $u_{\text{r}}(x,t) = h(x,0,t)$ is the temperature along the spine.
By \cref{thm:asymptotic_convergence_bdd_domain}, the solution to this equation asymptotically behaves as
\begin{equation}
   u_{\text{r}}(x,t)\approx  - \frac{2 \sqrt{ c_y}}{c_x^2} \frac{t^{-3/2}}{\Gamma(-1/2)} (-\Delta)^{-2}_x v(x). \label{eq:comb_rl_asymptotic}
\end{equation}
Now, suppose that we are interested in identifying the initial temperature from temperature measurements along the spine at a later time.  
Applying \cref{thm:backwards_problem}, we can reconstruct the initial data using the temperature measured at time $t$ with 
\begin{equation}
    v(x) \approx -t^{3/2} \Gamma(-1/2) \frac{c_x^2}{2 \sqrt{c_y}}(-\Delta_x)^2 u_{\text{r}}(x,t) \label{eq:comb_recon_RL}.
\end{equation}

    Alternatively, we can integrate out the $y$ dependence, measuring  along $x$ the temperature $ u_{\text{m}}(x,t) = \int h(x,y,t)dy$ marginalized over $y$ and 
      using 
    \begin{equation}
        \int \tilde H(k,y,s) dy =\, \tilde v(k) \frac{s^{-1/2}}{s^{1/2} + \frac{c_x}{2 \sqrt{c_y}}k^2 } 
        \, \, \xrightarrow \, \, 
        {\El^{-1}} \tilde v(k) e_{1/2,1}(t;k^2). \label{eq:H_marginalized}
    \end{equation}
    The marginal temperature $u_{\text{m}}(x,t)$ solves the Caputo subdiffusion problem of order $\alpha=1/2:$
\begin{equation}
  \left \{  \begin{aligned}
   & \caputo[]{1/2}{t} u_{\text{m}} - \frac{c_x}{2\sqrt{c_y}} \partial_x^2 u_{\text{m}} = 0
    \\
    &u_{\text{m}}(x,0)  =   v(x) .
    \end{aligned}  \right. 
\end{equation}
The marginal temperature asymptotically is given as: 
\begin{equation}
    u_{\text{m}}(x,t)  \approx \frac{2 \sqrt{c_y}}{c_x} \frac{t^{-1/2}}{\Gamma(1/2)} (-\Delta)^{-1}_x v(x). \label{eq:comb_caputo_asymptotic}
\end{equation}
The initial conditions can be reconstructed from the measured marginal temperature using
\begin{equation}
    v(x) \approx  t^{1/2} \Gamma(1/2) \frac{c_x}{2 \sqrt{c_y}}(-\Delta_x) u_{\text{m}}(x,t)\label{eq:comb_recon_caputo}.
\end{equation}
These two distinct representations of diffusion on a comb, $u_{\text{r}}$ and $u_{\text{m}}$, are shown in \cref{fig:comb_plots} for three different moments in time, along with the initial data reconstructed from the values of $u_{\text{r}}$ and $u_{\text{m}}$ measured at those times.

\begin{figure}
    \centering
    \includegraphics[width=0.98\linewidth]{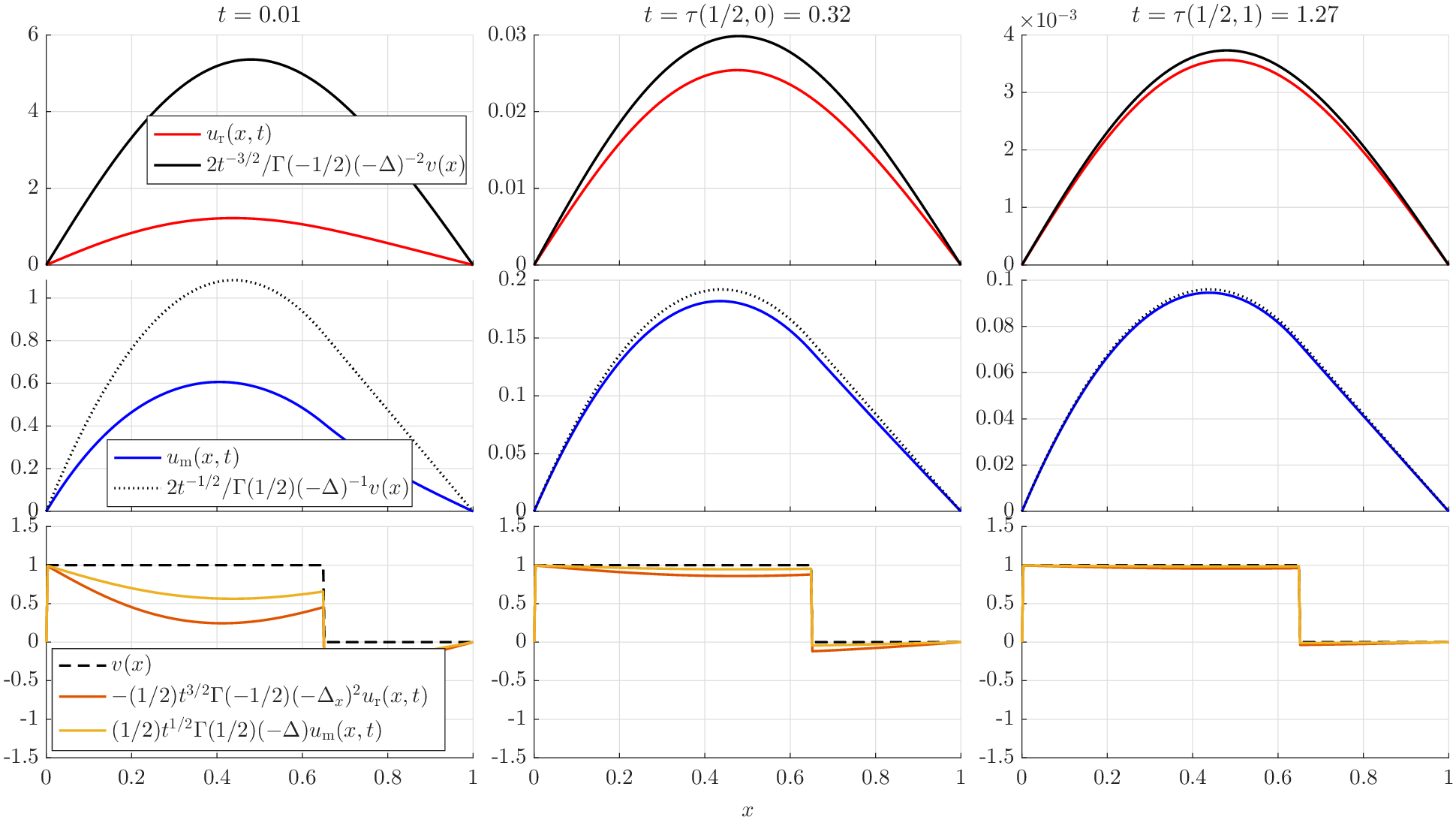}
    \caption{ 
    Temperature distribution along a comb at three different times: an early time ($t=0.01$) and the characteristic times for Caputo and Riemann-Liouville subdiffusion of order $1/2$ ($\tau(1/2,1) = 0.318 $ and $\tau(1/2,0) = 1.273$, respectively).
    The top row of plots shows the temperature $u_{\text{r}}(x,t)$ along the spine together with the asymptotic approximation given in \cref{eq:comb_rl_asymptotic} at these moments.
     The plots in the middle row show the temperature $u_{\text{m}}(x,t)$ marginalized over the ribs, and its asymptotic approximation \cref{eq:comb_caputo_asymptotic} for these three times. The bottom row shows the initial data reconstructed from $u_{\text {r}}$ and $u_{\text {m}}$ measured at these three times, as given in \cref{eq:comb_recon_RL,eq:comb_recon_caputo}.  Note that the scale on the $y$ axis is \textit{not} consistent between columns.
    }
    \label{fig:comb_plots}
\end{figure}

\section{Conclusion} \label{section:conclusion}
In this work, we demonstrate that solutions to the time-fractional diffusion equation,
\(
D^\alpha_tu - \Delta u=0,
\)
in subdiffusive and superdiffusive regimes, 
experience a 'freeze-out', factoring into spatial and temporal components at long times. The static long-time spatial profile is the inverse Laplacian or (negative) inverse bi-Laplacian of the initial data for Caputo and Riemann-Liouville derivatives respectively, while the temporal factor takes the form $\frac{t^{-\alpha + \beta}}{\Gamma(-\alpha+\beta)}$. This factorization results in a simple and robust method for recovering initial data from solution values taken at asymptotic times.
We extend this factorized representation to more general convolution memory kernels, demonstrating that initial-data preservation is not limited to time-fractional diffusion but holds in general (provided the memory kernel does not decay to zero in finite time).

We develop characteristic timescales for this factorization of the solutions, $\tau(\alpha,\beta)$ and $T(\alpha,\beta)$, and show that spatial features corresponding to frequency $k$ 'freeze' at times proportional to $\text{k}^{-2/\alpha}$. 
We show that the timescale $T(\alpha,\beta)$ is related to the number of real zeros $n_{\text{zeros}}$ of the Mittag-Leffler function, which permits the approximation $n_{\text{zeros}} \approx 2\frac{T(\alpha,\beta) \sin(\pi/\alpha)}{2 \pi}$.
We apply the results to diffusion on a comb, where restriction of transport to the fractal-like structure leads to both Caputo and Riemann-Liouville subdiffusion of order $\alpha=1/2.$

\bigskip
\noindent
{\bf {\large Acknowledgments}}
The authors acknowledge support from the 
Division of Mathematical Sciences at the US National Science 
Foundation (NSF) through Grants 
DMS-2111117 and DMS-2136198.




\bibliographystyle{cas-model2-names}

\bibliography{Fractional_Calculus.bib}

@article{abooaliComprehensiveStudyZeros2024,
  title = {A Comprehensive Study on the Zeros of the Two-Parameter {{Mittag-Leffler}} Function},
  author = {Abooali, Farnoosh and Jodayree Akbarfam, Aliasghar},
  year = 2024,
  month = dec,
  journal = {Sahand Communications in Mathematical Analysis},
  number = {Online First},
  publisher = {University of Maragheh},
  address = {IR},
  doi = {10.22130/scma.2024.2030271.1741},
  urldate = {2026-05-13},
  langid = {english}
}

@article{al-refaiGeneralisingFractionalCalculus2023,
  title = {Generalising the Fractional Calculus with {{Sonine}} Kernels via Conjugations},
  author = {{Al-Refai}, Mohammed and Fernandez, Arran},
  year = 2023,
  month = aug,
  journal = {Journal of Computational and Applied Mathematics},
  volume = {427},
  pages = {115159},
  issn = {03770427},
  doi = {10.1016/j.cam.2023.115159},
  urldate = {2026-04-10},
  langid = {english}
}

@article{aleroevFractionalSturmLiouvilleProblem2025,
  title = {The Fractional {{Sturm-Liouville}} Problem with the {{Caputo}} Derivative May Lose the Principal Eigenvalue},
  author = {Aleroev, Temirkhan S. and Li, Yulong},
  year = 2025,
  month = aug,
  journal = {Fractional Calculus and Applied Analysis},
  volume = {28},
  number = {4},
  pages = {1706--1716},
  issn = {1311-0454, 1314-2224},
  doi = {10.1007/s13540-025-00429-x},
  urldate = {2026-04-19},
  langid = {english}
}

@article{alimovBackwardProblemsTime2021,
  title = {On the Backward Problems in Time for Time-Fractional Subdiffusion Equations},
  author = {Alimov, Shavkat and Ashurov, Ravshan},
  year = 2021,
  journal = {Fractional Differential Calculus},
  volume = {11},
  number = {2},
  pages = {203--217},
  issn = {1847-9677},
  doi = {10.7153/fdc-2021-11-14},
  urldate = {2026-05-21},
  langid = {english}
}

@article{arkhincheevAnomalousDiffusionCharge2000,
  title = {Anomalous Diffusion and Charge Relaxation on Comb Model: Exact Solutions},
  shorttitle = {Anomalous Diffusion and Charge Relaxation on Comb Model},
  author = {Arkhincheev, V. E},
  year = 2000,
  month = jun,
  journal = {Physica A: Statistical Mechanics and its Applications},
  volume = {280},
  number = {3},
  pages = {304--314},
  issn = {0378-4371},
  doi = {10.1016/S0378-4371(99)00593-2},
  urldate = {2026-04-05}
}

@article{arkhincheevAnomalousDiffusionDrift1991,
  title = {Anomalous Diffusion and Drift in a Comb Model of Percolation Clusters},
  author = {Arkhincheev, V. E. and Baskin, E. and Tybulewicz, A.},
  year = 1991,
  journal = {Journal of Experimental and Theoretical Physics},
  urldate = {2026-05-20}
}

@article{armstrongAnomalousDiffusionFractal2025,
  title = {Anomalous Diffusion by Fractal Homogenization},
  author = {Armstrong, Scott and Vicol, Vlad},
  year = 2025,
  month = jun,
  journal = {Annals of PDE},
  volume = {11},
  number = {1},
  pages = {2},
  issn = {2524-5317, 2199-2576},
  doi = {10.1007/s40818-024-00189-6},
  urldate = {2026-08-13},
  langid = {english}
}

@article{bakuninMultiscalePercolationScaling2004,
  title = {Multi-Scale Percolation and Scaling Laws for Anisotropic Turbulent Diffusion},
  author = {Bakunin, O.G. and Schep, T.J.},
  year = 2004,
  month = feb,
  journal = {Physics Letters A},
  volume = {322},
  number = {1-2},
  pages = {105--110},
  issn = {03759601},
  doi = {10.1016/j.physleta.2003.10.082},
  urldate = {2026-08-04},
  langid = {english}
}

@article{benarousMultiscaleHomogenizationBounded2003,
  title = {Multiscale Homogenization with Bounded Ratios and Anomalous Slow Diffusion},
  author = {Ben Arous, G{\'e}rard and Owhadi, Houman},
  year = 2003,
  month = jan,
  journal = {Communications on Pure and Applied Mathematics},
  volume = {56},
  number = {1},
  pages = {80--113},
  issn = {0010-3640, 1097-0312},
  doi = {10.1002/cpa.10053},
  urldate = {2026-08-04},
  langid = {english}
}

@article{chengAsymptoticBehaviorSolutions2017,
  title = {Asymptotic Behavior of Solutions to Space-Time Fractional Diffusion Equations},
  author = {Cheng, Xing and Li, Zhiyuan and Yamamoto, Masahiro},
  year = 2017,
  month = mar,
  journal = {Mathematical Methods in the Applied Sciences},
  volume = {40},
  number = {4},
  eprint = {1505.06965},
  primaryclass = {math},
  pages = {1019--1031},
  issn = {0170-4214, 1099-1476},
  doi = {10.1002/mma.4033},
  urldate = {2026-04-05},
  archiveprefix = {arXiv}
}

@misc{cvetkoConvolutiontosumIdentitiesMittagLeffler2026,
  title = {Convolution-to-Sum Identities for {{Mittag-Leffler}} Type Functions},
  author = {Cvetko, William and Cherkaev, Elena},
  year = 2026,
  month = may,
  number = {arXiv:2605.01079},
  eprint = {2605.01079},
  primaryclass = {math},
  publisher = {arXiv},
  doi = {10.48550/arXiv.2605.01079},
  urldate = {2026-05-09},
  archiveprefix = {arXiv}
}

@article{eloeExistenceComparisonMonotonicity2026,
  title = {Existence, Comparison, and Monotonicity of Principal Eigenvalues of Fractional {{Sturm-Liouville}} Problems Involving {{Riemann-Liouville}} Derivatives},
  author = {Eloe, Paul W. and Li, Yulong},
  year = 2026,
  month = feb,
  journal = {Fractional Calculus and Applied Analysis},
  volume = {29},
  number = {1},
  pages = {1--25},
  issn = {1311-0454, 1314-2224},
  doi = {10.1007/s13540-026-00481-1},
  urldate = {2026-04-05},
  langid = {english}
}

@incollection{erdelyiMiscellaniousFunctions1955,
  title = {Miscellanious {{Functions}}},
  booktitle = {Higher {{Transcendental Functions}}},
  author = {Erdelyi, Arthur and Bateman, Harry and Magnus, Wilhelm and Oberhettinger, Fritz and Tricomi, Francesco G. and Project, Bateman Manuscript},
  year = 1955,
  volume = {3},
  pages = {206--227},
  publisher = {McGraw-Hill},
  address = {New York},
  urldate = {2026-04-05},
  langid = {english}
}

@article{floridiaBackwardProblemsTime2020,
  title = {Backward Problems in Time for Fractional Diffusion-Wave Equation},
  author = {Floridia, G and Yamamoto, M},
  year = 2020,
  month = dec,
  journal = {Inverse Problems},
  volume = {36},
  number = {12},
  pages = {125016},
  issn = {0266-5611, 1361-6420},
  doi = {10.1088/1361-6420/abbc5e},
  urldate = {2026-05-21}
}

@article{floridiaWellposednessBackwardProblems2020,
  title = {Well-Posedness for the Backward Problems in Time for General Time-Fractional Diffusion Equation},
  author = {Floridia, Giuseppe and Li, Zhiyuan and Yamamoto, Masahiro},
  year = 2020,
  month = oct,
  journal = {Rendiconti Lincei, Matematica e Applicazioni},
  volume = {31},
  number = {3},
  pages = {593--610},
  issn = {1120-6330, 1720-0768},
  doi = {10.4171/rlm/906},
  urldate = {2026-05-21}
}

@article{garrappaNumericalEvaluationTwo2015,
  title = {Numerical Evaluation of Two and Three Parameter {{Mittag-Leffler}} Functions},
  author = {Garrappa, Roberto},
  year = 2015,
  month = jan,
  journal = {SIAM Journal on Numerical Analysis},
  volume = {53},
  number = {3},
  pages = {1350--1369},
  issn = {0036-1429, 1095-7170},
  doi = {10.1137/140971191},
  urldate = {2026-04-05},
  langid = {english}
}

@misc{gorenfloFractionalCalculusIntegral2008,
  title = {Fractional Calculus: Integral and Differential Equations of Fractional Order},
  shorttitle = {Fractional {{Calculus}}},
  author = {Gorenflo, Rudolf and Mainardi, Francesco},
  year = 2008,
  month = may,
  number = {arXiv:0805.3823},
  eprint = {0805.3823},
  primaryclass = {math-ph},
  publisher = {arXiv},
  doi = {10.48550/arXiv.0805.3823},
  urldate = {2026-04-20},
  archiveprefix = {arXiv}
}

@book{gorenfloMittagLefflerFunctionsRelated2014,
  title = {Mittag-{{Leffler Functions}}, {{Related Topics}} and {{Applications}}: {{Theory}} and {{Applications}}},
  shorttitle = {Mittag-{{Leffler Functions}}, {{Related Topics}} and {{Applications}}},
  author = {Gorenflo, Rudolf and Kilbas, Anatoly A. and Mainardi, Francesco and Rogosin, Sergei V.},
  year = 2014,
  series = {Springer {{Monographs}} in {{Mathematics}}},
  publisher = {Springer Berlin Heidelberg},
  address = {Berlin, Heidelberg},
  doi = {10.1007/978-3-662-43930-2},
  urldate = {2026-04-19},
  copyright = {https://www.springernature.com/gp/researchers/text-and-data-mining},
  isbn = {978-3-662-43929-6 978-3-662-43930-2},
  langid = {english}
}

@article{gorenfloRecentAdvancesTheory2009,
  title = {Some Recent Advances in Theory and Simulation of Fractional Diffusion Processes},
  author = {Gorenflo, Rudolf and Mainardi, Francesco},
  year = 2009,
  month = jul,
  journal = {Journal of Computational and Applied Mathematics},
  volume = {229},
  number = {2},
  pages = {400--415},
  issn = {03770427},
  doi = {10.1016/j.cam.2008.04.005},
  urldate = {2026-05-21},
  copyright = {https://www.elsevier.com/tdm/userlicense/1.0/},
  langid = {english}
}

@article{grigolettoLinearFractionalDifferential2018,
  title = {Linear Fractional Differential Equations and Eigenfunctions of Fractional Differential Operators},
  author = {Grigoletto, Eliana Contharteze and De Oliveira, Edmundo Capelas and De Figueiredo Camargo, Rubens},
  year = 2018,
  month = may,
  journal = {Computational and Applied Mathematics},
  volume = {37},
  number = {2},
  pages = {1012--1026},
  issn = {0101-8205, 1807-0302},
  doi = {10.1007/s40314-016-0381-1},
  urldate = {2026-04-05},
  langid = {english}
}

@incollection{hannekenEnumerationRealZeros2007,
  title = {Enumeration of the Real Zeros of the {{Mittag-Leffler}} Function {{E$\alpha$}}(z), 1 {$<\alpha<$} 2},
  booktitle = {Advances in {{Fractional Calculus}}},
  author = {Hanneken, John W. and Vaught, David M. and Achar, B. N. Narahari},
  editor = {Sabatier, Jocelyn and Agrawal, Om Prakash and Machado, J. A. Tenreiro},
  year = 2007,
  pages = {15--26},
  publisher = {Springer Netherlands},
  address = {Dordrecht},
  doi = {10.1007/978-1-4020-6042-7_2},
  urldate = {2026-04-05},
  isbn = {978-1-4020-6041-0 978-1-4020-6042-7},
  langid = {english}
}

@article{hauboldMittagLefflerFunctionsTheir2011,
  title = {Mittag-{{Leffler}} Functions and Their Applications},
  author = {Haubold, H. J. and Mathai, A. M. and Saxena, R. K.},
  editor = {Tsitouras, Ch},
  year = 2011,
  month = jan,
  journal = {Journal of Applied Mathematics},
  volume = {2011},
  number = {1},
  pages = {298628},
  issn = {1110-757X, 1687-0042},
  doi = {10.1155/2011/298628},
  urldate = {2026-05-21},
  langid = {english}
}

@article{heymansPhysicalInterpretationInitial2006,
  title = {Physical Interpretation of Initial Conditions for Fractional Differential Equations with {{Riemann-Liouville}} Fractional Derivatives},
  author = {Heymans, Nicole and Podlubny, Igor},
  year = 2006,
  month = jun,
  journal = {Rheologica Acta},
  volume = {45},
  number = {5},
  pages = {765--771},
  issn = {0035-4511, 1435-1528},
  doi = {10.1007/s00397-005-0043-5},
  urldate = {2026-04-05},
  copyright = {http://www.springer.com/tdm},
  langid = {english}
}

@article{hilferFractionalDiffusionBased2000,
  title = {Fractional Diffusion Based on {{Riemann-Liouville}} Fractional Derivatives},
  author = {Hilfer, R.},
  year = 2000,
  month = apr,
  journal = {The Journal of Physical Chemistry B},
  volume = {104},
  number = {16},
  pages = {3914--3917},
  issn = {1520-6106, 1520-5207},
  doi = {10.1021/jp9936289},
  urldate = {2026-06-16},
  langid = {english}
}

@book{iominFractionalDynamicsComblike2018,
  title = {Fractional {{Dynamics}} in {{Comb-like Structures}}},
  author = {Iomin, Alexander and M{\'e}ndez, Vicen{\c c} and Horsthemke, Werner},
  year = 2018,
  month = oct,
  publisher = {World Scientific},
  doi = {10.1142/11076},
  urldate = {2026-04-05},
  isbn = {978-981-327-343-6 978-981-327-344-3},
  langid = {english}
}

@article{jinTutorialInverseProblems2015,
  title = {A Tutorial on Inverse Problems for Anomalous Diffusion Processes},
  author = {Jin, Bangti and Rundell, William},
  year = 2015,
  month = mar,
  journal = {Inverse Problems},
  volume = {31},
  number = {3},
  eprint = {1501.00251},
  primaryclass = {math},
  pages = {035003},
  issn = {0266-5611, 1361-6420},
  doi = {10.1088/0266-5611/31/3/035003},
  urldate = {2026-04-20},
  archiveprefix = {arXiv}
}

@misc{kemppainenDecayEstimatesTimefractional2014,
  title = {Decay Estimates for Time-Fractional and Other Non-Local in Time Subdiffusion Equations in $\mathbb{R}^d$},
  author = {Kemppainen, Jukka and Siljander, Juhana and Vergara, Vicente and Zacher, Rico},
  year = 2014,
  month = mar,
  number = {arXiv:1403.1737},
  eprint = {1403.1737},
  primaryclass = {math},
  publisher = {arXiv},
  doi = {10.48550/arXiv.1403.1737},
  urldate = {2026-04-05},
  archiveprefix = {arXiv}
}

@book{kilbasTheoryApplicationsFractional2006,
  title = {Theory and Applications of Fractional Differential Equations},
  author = {Kilbas, Anatoly A.},
  year = 2006,
  series = {North-{{Holland}} Mathematics Studies},
  number = {204},
  publisher = {Elsevier},
  address = {Boston},
  collaborator = {Srivastava, Hari M. and Trujillo, Juan J.},
  isbn = {978-0-444-51832-3 978-0-08-046207-3},
  langid = {english}
}

@article{kimAsymptoticBehaviorsFundamental2016,
  title = {Asymptotic Behaviors of Fundamental Solution and Its Derivatives Related to Space-Time Fractional Differential Equations},
  author = {Kim, Kyeong-Hun and Lim, Sungbin},
  year = 2016,
  month = jul,
  journal = {Journal of the Korean Mathematical Society},
  volume = {53},
  number = {4},
  pages = {929--967},
  publisher = {대한수학회},
  doi = {10.4134/JKMS.J150343},
  urldate = {2026-05-21}
}

@article{liAsymptoticsSolutionsSuperdiffusion2023,
  title = {On Asymptotics of Solutions for Superdiffusion and Subdiffusion Equations with the {{Riemann-Liouville}} Fractional Derivative},
  author = {Li, Zhiqiang and Fan, Yanzhe and Li, Zhiqiang and Fan, Yanzhe},
  year = 2023,
  journal = {AIMS Mathematics},
  volume = {8},
  number = {8},
  pages = {19210--19239},
  issn = {2473-6988},
  doi = {10.3934/math.2023980},
  urldate = {2026-04-05},
  copyright = {2023 The Author(s)},
  langid = {english}
}

@article{liInitialboundaryValueProblems2023,
  title = {Initial-Boundary Value Problems for Coupled Systems of Time-Fractional Diffusion Equations},
  author = {Li, Zhiyuan and Huang, Xinchi and Liu, Yikan},
  year = 2023,
  month = apr,
  journal = {Fractional Calculus and Applied Analysis},
  volume = {26},
  number = {2},
  pages = {533--566},
  issn = {1311-0454, 1314-2224},
  doi = {10.1007/s13540-023-00149-0},
  urldate = {2026-04-05},
  langid = {english}
}

@article{liuBackwardProblemTimefractional2010,
  title = {A Backward Problem for the Time-Fractional Diffusion Equation},
  author = {Liu, J.J. and Yamamoto, M.},
  year = 2010,
  month = nov,
  journal = {Applicable Analysis},
  volume = {89},
  number = {11},
  pages = {1769--1788},
  issn = {0003-6811, 1563-504X},
  doi = {10.1080/00036810903479731},
  urldate = {2026-04-20},
  langid = {english}
}

@article{lorenzoGeneralizedFunctionsFractional2008,
  title = {Generalized Functions for the Fractional Calculus},
  author = {Lorenzo, Carl F. and Hartley, Tom T.},
  year = 2008,
  journal = {Critical Reviews in Biomedical Engineering},
  volume = {36},
  number = {1},
  pages = {39--55},
  issn = {0278-940X},
  doi = {10.1615/CritRevBiomedEng.v36.i1.40},
  urldate = {2026-05-13},
  langid = {english}
}

@article{maAsymptoticsSolutionsAnomalous2013,
  title = {The Asymptotics of the Solutions to the Anomalous Diffusion Equations},
  author = {Ma, Yutian and Zhang, Fengrong and Li, Changpin},
  year = 2013,
  month = sep,
  journal = {Computers \& Mathematics with Applications},
  series = {Fractional {{Differentiation}} and Its {{Applications}}},
  volume = {66},
  number = {5},
  pages = {682--692},
  issn = {0898-1221},
  doi = {10.1016/j.camwa.2013.01.032},
  urldate = {2026-04-05}
}

@book{mainardiFractionalCalculusWaves2010,
  title = {Fractional Calculus and Waves in Linear Viscoelasticity: An Introduction to Mathematical Models},
  shorttitle = {Fractional Calculus and Waves in Linear Viscoelasticity},
  author = {Mainardi, F.},
  year = 2010,
  publisher = {Imperial College Press},
  address = {London ; Hackensack, NJ},
  isbn = {978-1-84816-329-4},
  langid = {english},
  lccn = {QA314 .M35 2010}
}

@article{mainardiFractionalRelaxationoscillationFractional1996,
  title = {Fractional Relaxation-Oscillation and Fractional Diffusion-Wave Phenomena},
  author = {Mainardi, Francesco},
  year = 1996,
  month = sep,
  journal = {Chaos, Solitons \& Fractals},
  volume = {7},
  number = {9},
  pages = {1461--1477},
  issn = {09600779},
  doi = {10.1016/0960-0779(95)00125-5},
  urldate = {2026-04-05},
  langid = {english}
}

@article{mainardiMittagLefflertypeFunctionsFractional2000,
  title = {On {{Mittag-Leffler-type}} Functions in Fractional Evolution Processes},
  author = {Mainardi, Francesco and Gorenflo, Rudolf},
  year = 2000,
  month = jun,
  journal = {Journal of Computational and Applied Mathematics},
  volume = {118},
  number = {1-2},
  pages = {283--299},
  issn = {03770427},
  doi = {10.1016/S0377-0427(00)00294-6},
  urldate = {2026-04-05},
  copyright = {https://www.elsevier.com/tdm/userlicense/1.0/},
  langid = {english}
}

@article{mainardiPropertiesMittagLefflerFunction2014,
  title = {On some properties of the {Mittag-Leffler} function ${E_\alpha(-t^\alpha)}$, completely monotone for ${t> 0}$ with ${0<\alpha<1}$},
  author = {Mainardi, Francesco},
  year = 2014,
  journal = {Discrete and Continuous Dynamical Systems - B},
  volume = {19},
  number = {7},
  pages = {2267-2278},
  issn = {1531-3492, 1553-524X},
  doi = {10.3934/dcdsb.2014.19.2267},
  urldate = {2026-04-05}
}

@article{mainardiWhyMittagLefflerFunction2020,
  title = {Why the {{Mittag-Leffler}} Function Can Be Considered the Queen Function of the Fractional Calculus?},
  author = {Mainardi, Francesco},
  year = 2020,
  month = nov,
  journal = {Entropy},
  volume = {22},
  number = {12},
  pages = {1359},
  issn = {1099-4300},
  doi = {10.3390/e22121359},
  urldate = {2026-04-05},
  langid = {english}
}

@article{mendezComblikeModelsTransport2013,
  title = {Comb-like Models for Transport along Spiny Dendrites},
  author = {M{\'e}ndez, Vicen{\c c} and Iomin, Alexander},
  year = 2013,
  month = aug,
  journal = {Chaos, Solitons \& Fractals},
  volume = {53},
  pages = {46--51},
  issn = {09600779},
  doi = {10.1016/j.chaos.2013.05.002},
  urldate = {2026-08-04},
  langid = {english}
}

@article{metzlerRandomWalksGuide2000,
  title = {The Random Walk's Guide to Anomalous Diffusion: A Fractional Dynamics Approach},
  shorttitle = {The Random Walk's Guide to Anomalous Diffusion},
  author = {Metzler, Ralf and Klafter, Joseph},
  year = 2000,
  month = dec,
  journal = {Physics Reports},
  volume = {339},
  number = {1},
  pages = {1--77},
  issn = {03701573},
  doi = {10.1016/S0370-1573(00)00070-3},
  urldate = {2026-04-05},
  copyright = {https://www.elsevier.com/tdm/userlicense/1.0/},
  langid = {english}
}

@misc{mieghemMittagLefflerFunction2021,
  title = {The {{Mittag-Leffler}} Function},
  author = {Mieghem, Piet Van},
  year = 2021,
  month = sep,
  number = {arXiv:2005.13330},
  eprint = {2005.13330},
  primaryclass = {math},
  publisher = {arXiv},
  doi = {10.48550/arXiv.2005.13330},
  urldate = {2026-04-11},
  archiveprefix = {arXiv},
  langid = {english}
}

@article{milovanovTurbulenceSpreadingAnomalous2025,
  title = {Turbulence Spreading and Anomalous Diffusion on Combs},
  author = {Milovanov, Alexander V. and Iomin, Alexander and Rasmussen, Jens Juul},
  year = 2025,
  month = jun,
  journal = {Physical Review E},
  volume = {111},
  number = {6},
  pages = {064217},
  issn = {2470-0045, 2470-0053},
  doi = {10.1103/cmf5-sf8x},
  urldate = {2026-08-04},
  langid = {english}
}

@book{parisAsymptoticsMellinBarnesIntegrals2001,
  title = {Asymptotics and {{Mellin-Barnes Integrals}}},
  author = {Paris, R. B. and Kaminski, D.},
  year = 2001,
  series = {Encyclopedia of {{Mathematics}} and Its {{Applications}}},
  publisher = {Cambridge University Press},
  address = {Cambridge},
  doi = {10.1017/CBO9780511546662},
  urldate = {2026-04-05},
  isbn = {978-0-521-79001-7}
}

@article{raghavanFractionalDerivativesApplication2011,
  title = {Fractional Derivatives: {{Application}} to Transient Flow},
  shorttitle = {Fractional Derivatives},
  author = {Raghavan, R.},
  year = 2011,
  month = dec,
  journal = {Journal of Petroleum Science and Engineering},
  volume = {80},
  number = {1},
  pages = {7--13},
  issn = {09204105},
  doi = {10.1016/j.petrol.2011.10.003},
  urldate = {2026-08-04},
  langid = {english}
}

@article{sakamotoInitialValueBoundary2011,
  title = {Initial Value/Boundary Value Problems for Fractional Diffusion-Wave Equations and Applications to Some Inverse Problems},
  author = {Sakamoto, Kenichi and Yamamoto, Masahiro},
  year = 2011,
  month = oct,
  journal = {Journal of Mathematical Analysis and Applications},
  volume = {382},
  number = {1},
  pages = {426--447},
  issn = {0022247X},
  doi = {10.1016/j.jmaa.2011.04.058},
  urldate = {2026-04-20},
  copyright = {https://www.elsevier.com/tdm/userlicense/1.0/},
  langid = {english}
}

@article{sandevGeneralizedDiffusionwaveEquation2019,
  title = {Generalized Diffusion-Wave Equation with Memory Kernel},
  author = {Sandev, Trifce and Tomovski, Zivorad and Dubbeldam, Johan L A and Chechkin, Aleksei},
  year = 2019,
  month = jan,
  journal = {Journal of Physics A: Mathematical and Theoretical},
  volume = {52},
  number = {1},
  pages = {015201},
  issn = {1751-8113, 1751-8121},
  doi = {10.1088/1751-8121/aaefa3},
  urldate = {2026-04-10},
  langid = {english}
}

@article{sandevHeterogeneousDiffusionComb2018,
  title = {Heterogeneous Diffusion in Comb and Fractal Grid Structures},
  author = {Sandev, Trifce and Schulz, Alexander and Kantz, Holger and Iomin, Alexander},
  year = 2018,
  month = sep,
  journal = {Chaos, Solitons \& Fractals},
  volume = {114},
  pages = {551--555},
  issn = {09600779},
  doi = {10.1016/j.chaos.2017.04.041},
  urldate = {2026-08-12},
  langid = {english}
}

@article{zhangNumericalAnalysisBackward2020,
  title = {Numerical Analysis of Backward Subdiffusion Problems},
  author = {Zhang, Zhengqi and Zhou, Zhi},
  year = 2020,
  month = oct,
  journal = {Inverse Problems},
  volume = {36},
  number = {10},
  pages = {105006},
  issn = {0266-5611, 1361-6420},
  doi = {10.1088/1361-6420/abaf3d},
  urldate = {2026-05-13}
}



\end{document}